\documentclass[10pt]{amsart}
\usepackage{graphicx, amsmath, fullpage, amssymb, amsthm, amsfonts,color,mathrsfs,moreverb, mathtools,float}
\usepackage{enumerate}

\newtheorem{theorem}{Theorem}[section]
\newtheorem{lemma}[theorem]{Lemma}

\newtheorem{corollary}[theorem]{Corollary}

\theoremstyle{definition}
\newtheorem{definition}[theorem]{Definition}

\usepackage[toc,page]{appendix} 
\usepackage[svgnames]{xcolor}
\theoremstyle{remark}
\newtheorem{remark}[theorem]{Remark}
\numberwithin{equation}{section}
\usepackage[colorlinks,
            linkcolor=FireBrick,
            anchorcolor=red,
            citecolor=RoyalBlue
            ]{hyperref}
 
\begin{document}

\title{Sparse Random Tensors: Concentration, Regularization and Applications}
\author{Zhixin Zhou}
\address{ Department of Management Sciences, City University of Hong Kong }
\email{zhixzhou@cityu.edu.hk  }

\author{Yizhe Zhu}
\address{Department of Mathematics, University of California, San Diego, La Jolla, CA 92093}
\email{yiz084@ucsd.edu}

%
\date{\today}

\begin{abstract}  
We prove a non-asymptotic concentration inequality for the spectral norm of sparse inhomogeneous random tensors with Bernoulli entries. For an order-$k$ inhomogeneous random tensor $T$ with  sparsity $p_{\max}\geq \frac{c\log n}{n }$, we show that $\|T-\mathbb E T\|=O(\sqrt{n p_{\max}}\log^{k-2}(n))$ with high probability. The optimality of this bound up to polylog  factors is provided by an information theoretic lower bound.
By tensor unfolding, we extend the range of sparsity to 
$p_{\max}\geq  \frac{c\log n}{n^{m}}$ with $1\leq m\leq k-1$ and obtain concentration inequalities for different sparsity regimes.
We also provide a simple way to regularize  $T$ such that $O(\sqrt{n^{m}p_{\max}})$ concentration still holds  down to sparsity  $p_{\max}\geq \frac{c}{n^{m}}$ with $k/2\leq m\leq k-1$.  We present our concentration and regularization results with two applications: (i) a randomized construction of  hypergraphs of bounded degrees with  good expander mixing properties,  (ii) concentration of sparsified tensors under uniform sampling.
\end{abstract}

\maketitle


\section{Introduction}
\label{sec1}

Tensors have been popular data formats in machine learning and network
analysis. The statistical models on tensors and the related algorithms have
been widely studied in the last ten years, including tensor decomposition
\cite{anandkumar2014tensor,ge2015escaping}, tensor completion
\cite{jain2014provable,montanari2018spectral}, tensor sketching
\cite{nguyen2015tensor}, tensor PCA
\cite{richard2014statistical,chen2019phase,arous2019landscape}, and community
detection on hypergraphs
\cite{kim2017community,ghoshdastidar2017consistency,pal2019community,cole2020exact}.
This raises the urgent demand for random tensor theory, especially the
concentration inequalities in a non-asymptotic point of view. There are
several concentration results of sub-Gaussian random tensors
\cite{tomioka2014spectral} and Gaussian tensors
\cite{auffinger2013random,richard2014statistical,nguyen2015tensor}. Recently
concentration inequalities for rank-$1$ tensor were also studied in
\cite{vershynin2019concentration} with application to the capacity of polynomial
threshold functions \cite{baldi2019polynomial}. In many of the applications
in data science, the sparsity of the random tensor is an important aspect.
However, there are only a few results for the concentration of order-$3$
sparse random tensors \cite{jain2014provable,lei2019consistent}, and not
much is known for general order-$k$ sparse random tensors.

Inspired by discrepancy properties in random hypergraph theory, we prove
concentration inequalities on sparse random tensors in the measurement
of the tensor spectral norm. Previous results for tensors include the concentration
of sub-Gaussian tensors and expectation bound on the spectral norm for
general random tensors \cite{tomioka2014spectral,nguyen2015tensor}. The
sparsity parameter does not appear in those bounds and directly applying
those results would not get the desired concentration for sparse random
tensors.

To simplify our presentation, we focus on real-valued order-$k$
$n\times \dots \times n$ tensors, while the results can be extended to
tensors with other dimensions. We denote the set of these tensors by
$\mathbb{R}^{n^{k}}$. We first define the Frobenius inner product and spectral
norm for tensors.

\begin{definition}[Frobenius inner product and spectral norm]
For order-$k$ $n\times \dots \times n$ tensors $T$ and $A$, the
\textit{Frobenius inner product} is defined by the sum of entrywise products:
\begin{align*}
\langle T, A\rangle := \sum _{i_{1},\dots ,i_{k}\in [n]} t_{i_{1},
\dots , i_{k}} a_{i_{1},\dots ,i_{k}},
\end{align*}
and the \textit{Frobenius norm} is defined by
$\|T\|_{F} := \sqrt{\langle T,T\rangle }$. Let
$x_{1}\otimes \dots \otimes x_{k}\in \mathbb{R}^{n^{k}}$ be the outer product
of vectors $x_{1},\dots , x_{k}\in \mathbb{R}^{n}$, i.e.,
\begin{equation*}
(x_{1}\otimes \dots \otimes x_{k})_{i_{1},\dots ,i_{k}}=x_{1,i_{1}}
\cdots x_{k,i_{k}}
\end{equation*}
for $i_{1},\dots ,i_{k}\in [n]$. Then the \textit{spectral norm} of
$T$ is defined by
\begin{align*}
\|T\|:&=\sup _{\|x_{1}\|_{2}=\cdots = \|x_{k}\|_{2}=1}|\langle T, x_{1}
\otimes \dots \otimes x_{k}\rangle |
\\
& = \sup _{\|x_{1}\|_{2}=\cdots = \|x_{k}\|_{2}=1} \left |\sum _{i_{1},
\dots ,i_{k}\in [n]} t_{i_{1},\dots ,i_{k}} x_{1,i_{1}}\cdots x_{k,i_{k}}
\right |.
\end{align*}
\end{definition}

For the ease of notation, we denote the Frobenius inner product between
a tensor $T$ and a tensor $x_{1}\otimes \dots \otimes x_{k}$ by
\begin{align*}
T(x_{1},\dots ,x_{k}):=\langle T,x_{1}\otimes \dots \otimes x_{k}
\rangle ,
\end{align*}
which can be seen as a multi-linear form on $x_{1},\dots , x_{k}$. It is
worth noting that $\|T\|\leq \|T\|_{F}$ since the following inequality
holds:
%
\begin{align}
\label{eq:Fnorm}
\|T\|& = \sup _{\|x_{1}\|_{2}=\cdots = \|x_{k}\|_{2}=1} |T(x_{1},
\dots ,x_{k})|\le \sup _{A:\|A\|_{F}\le 1}|\langle T,A\rangle | = \|T
\|_{F}.
\end{align}

In general, it is NP-hard to compute the spectral norm of tensors for
$k\geq 3$ \cite{hillar2013most}. However, it would be possible to show
the concentration of sparse random tensors in the measurement of the spectral
norm with high probability.

There have been many fruitful results on the concentration of random matrices,
including the sparse ones. We briefly discuss different proof techniques
and their difficulty and limitation for generalization to random tensors.
For sub-Gaussian matrices, an $\varepsilon $-net argument will quickly
give a desired spectral norm bound \cite{vershynin2018high}. For Gaussian
matrices, one could relate the spectral norm to the maximal of a certain
Gaussian process \cite{van2017spectral}. Another powerful way to derive
a good spectral norm bound for random matrices is called the high moment
method. Considering a centered $n\times n$ Hermitian random matrix
$A$, for any integer $k$, its spectral norm satisfies
$\mathbb{E}[\|A\|^{2k}]\leq \mathbb{E}[\text{tr}(A^{2k})]$. By taking
$k$ growing with $n$, if one can have a good estimate of
$\mathbb{E}[\textnormal{tr}(A^{2k})]$, it implies a good concentration bound
on $\|A\|$. It's well-known that computing the trace of a random matrix
is equivalent to counting a certain class of cycles in a graph. This type
of argument, together with some more refined modifications and variants
(e.g. estimating high moments for the corresponding non-backtracking operator),
is particularly useful for bounding the spectral norm of sparse random
matrices, see
\cite{vu2007spectral,bandeira2016sharp,benaych2017spectral,latala2018dimension,bordenave2018nonbacktracking}.

A different approach is called the Kahn-Szemer{\'{e}}di argument, which was
first applied to obtain the spectral gap of random regular graphs
\cite{friedman1989}, and was extensively applied to other random graph
models
\cite{broder1998optimal,lubetzky2011spectra,dumitriu2013functional,cook2018size,tikhomirov2019spectral,zhu2020second}.
In particular, this argument was used in
\cite{feige2005spectral,lei2015consistency} to estimate the largest eigenvalue
of sparse Erd\H{o}s-R\'{e}nyi graphs. Although using this method one cannot
obtain the exact constant of the spectral norm, it does capture the right
order on $n$ and the sparsity parameter $p$.

A natural question is how those methods can be applied to study the spectral
norm of random tensors. For sub-Gaussian random tensors of order $k$, the
$\varepsilon $-net argument would give us a spectral norm bound
$O(\sqrt{n})$ \cite{tomioka2014spectral}. However, the dependence on the
order $k$ might not be optimal, and it cannot capture the sparsity in the
sparse random tensor case. For Gaussian random tensors, surprisingly, none
of the above approaches could obtain a sharp spectral norm bound with the
correct constant. Instead, the exact asymptotic spectral norm was given
in \cite{auffinger2013random} using techniques from spin glasses. This
is also the starting point for a line of further research: tensor PCA and
spiked tensor models under Gaussian noise, see for example
\cite{richard2014statistical,lesieur2017statistical,chen2019phase,arous2019landscape}.
The tools from spin glasses rely heavily on the assumption of Gaussian
distribution and cannot be easily adapted to non-Gaussian cases.

One might try to develop a high moment method for random tensors. Unfortunately,
there is no natural generalization of the trace or eigenvalues for tensors
that match our cycle counting interpretation in the random matrix case.
Instead, by projecting the random tensor into a matrix form (including
the adjacency matrix, self-avoiding matrix, and the non-backtracking matrix
of a hypergraph), one could still apply the moment method to obtain some
information of the original tensor or hypergraph, see
\cite{lu2012loose,pal2019community,dumitriu2019spectra,angelini2015spectral}.
This approach is particularly useful for the study of community detection
problems on random hypergraphs. However, after reducing the adjacency tensor
into an adjacency matrix, there is a strict information loss and one could
not get the exact spectral norm information of the original tensor. Due
to the barrier of extending other methods to sparse random tensors, we
generalize the Kahn-Szemer{\'{e}}di argument to obtain a good spectral norm
bound when $p\geq \frac{c\log n}{n}$.

\section{Main results}
\label{sec2}

\subsection{Concentration}
\label{sec2.1}

Let $P=(p_{i_{1},\dots ,i_{k}})\in [0,1]^{n^{k}}$ be an order-$k$ tensor
and $T$ be a random tensor with independent entries such that
%
\begin{align}
\label{eq:assumptionT}
t_{i_{1},\dots ,i_{k}}\sim \text{Bernoulli}(p_{i_{1},\dots ,i_{k}}),
\quad \text{ where in particular, } P=\mathbb{E} T.
\end{align}
To control the sparsity of random tensor, we introduce the parameter for
maximal probability
\begin{equation*}
p:=p_{\max } :=\max _{i_{1},\dots ,i_{k}\in [n]} p_{i_{1},\dots ,i_{k}}.
\end{equation*}
Note that when $k=2$, $np$ is the maximal expected degree parameter in
\cite{feige2005spectral,lei2015consistency,le2017concentration}. For two
order-$k$ tensors $A,B\in \mathbb{R}^{n^{k}}$, define the Hadamard product
$A\circ B$ as
\begin{equation*}
(A\circ B)_{i_{1},\dots ,i_{k}}:=A_{i_{1},\dots ,i_{k}}B_{i_{1},
\dots ,i_{k}}.
\end{equation*}
Now we are ready to state our first main result, which is a generalization
of the case when $k=2$ in \cite{feige2005spectral,lei2015consistency} to
all $k\geq 2$.

\begin{theorem}
\label{thm:lei}
Let $k\geq 2$ be fixed. Let $A$ be a deterministic tensor of order
$k$, and $T$ be a random tensor of order $k$ with Bernoulli entries. Assume
$p\geq \frac{c \log n}{n} $ for some constant $c>0$. Then for any
$r>0$, there is a constant $C>0$ depending only on $r,c,k$ such that with
probability at least $1-n^{-r}$,
%
\begin{align}
\label{eq:Hadamardproductinequality}
\| A\circ T-\mathbb{E}[A\circ T] \|\leq C\sqrt{n p}\log ^{k-2} (n)
\max _{i_{1},\dots ,i_{k}}|A_{i_{1},\dots ,i_{k}}|.
\end{align}
\end{theorem}

\begin{remark}
From the proof of Theorem~\ref{thm:lei} in Section~\ref{sec:prooflei},
the constant $C$ in~\eqref{eq:Hadamardproductinequality} depends exponentially
on $k$ and linearly on $r$.
\end{remark}

In fact, we can provide a lower bound on the spectral norm as follows,
which shows~\eqref{eq:Hadamardproductinequality} is tight up to a polylog
factor.

\begin{theorem}%
\label{thm:lowerboundSpec}
Let $k\geq 2$ be fixed and $T$ be a random tensor of order $k$ with Bernoulli
entries and $\mathbb{E} T_{i_{1},\dots ,i_{k}}=p$ for
$i_{1},\dots ,i_{k}\in [n]$. Assume $p=o(1)$ and $np\to \infty $ as
$n\to \infty $. Then with high probability,
\begin{equation*}
\| T-\mathbb{E}T \|\geq \sqrt{n p}.
\end{equation*}
\end{theorem}

From~\eqref{eq:Fnorm}, when $p\geq \frac{c\log n}{n^{k-1}}$, a concentration
bound by Bernstein inequality on $\|T-\mathbb{E}T\|_{F}$ implies
$\|T-\mathbb{E} T\|=O(\sqrt{n^{k}p})$ with high probability. Applying tensor
unfolding, we can improve this bound and obtain concentration inequalities
for different sparsity ranges as follows.

\begin{theorem}
\label{thm:main}
Let $k\geq 2$ be fixed and $T$ be a random tensor defined in~\eqref{eq:assumptionT}. Assume $ p\geq \frac{c \log n}{n^{m}} $ for some
constant $c>0$ and an integer $m$ such that $k/2\leq m\leq k-1$. Then for
any $r>0$, there is a constant $C>0$ depending only on $r,c,k$ such that
with probability at least $1-n^{-r}$,
%
\begin{align}
\label{eq:unfoldeq1}
\| T-\mathbb{E}T \|\leq C\sqrt{n^{m} p}.
\end{align}
Assume $p\geq \frac{c \log n}{n^{m}}$ with $1\leq m< k/2$. Then there is
a constant $C>0$ depending on $r,c,k$ such that with probability at least
$1-n^{-r}$,
%
\begin{align}
\label{eq:unfoldeq2}
\|T-\mathbb{E}T\|\leq
\begin{cases}
C\sqrt{n^{m}p} \log ^{\frac{k-1}{m}-1}(n) & \text{if } k/m\not \in
\mathbb{Z},
\\
C\sqrt{n^{m}p} \log ^{\frac{k}{m}-2}(n) & \text{if } k/m \in \mathbb{Z}.
\end{cases}
\end{align}
\end{theorem}

\begin{remark}%
\label{rmk:fiber}
In~\eqref{eq:Hadamardproductinequality}, the factor $\sqrt{np}$ corresponds
to the Euclidean norm of a fiber (a vector obtained by fixing all but one
indices in a tensor) of $T$. The same applies in Theorem~\ref{thm:main} after tensor unfolding. The factor $\sqrt{n^{m}p}$ in~\eqref{eq:unfoldeq1} corresponds to the Euclidean norm of a row in the
unfolded $n^{k-m} \times n^{m}$ matrix form of $ T$. Similarly, the same
factor in~\eqref{eq:unfoldeq2} corresponds to the Euclidean norm of an
$n^{m}$-dimensional fiber in an unfolded tensor form of $T$.
\end{remark}

\begin{remark}
When $p=\frac{c\log n}{n^{m}}$ with $1\leq m\leq k-1$, the inequalities
in Theorem~\ref{thm:main} are tight up to polylog factors. Note that~\eqref{eq:unfoldeq1} and~\eqref{eq:unfoldeq2} imply
$\|T-\mathbb{E}T \|\leq C\sqrt{c} \log ^{k-1/2}(n)$. On the other hand,
when $\frac{c\log n}{n^{k-1}}\leq p\leq \frac{1}{2}$, with high probability
the random tensor $T$ has at least one non-zero entry. By the inequality
$\|T-\mathbb{E}T\|\geq \max _{i_{1},\dots ,i_{k}}|T_{i_{1},\dots ,i_{k}}-p_{i_{1},
\dots ,i_{k}}|$, we have $\|T-\mathbb{E}T\|\geq 1/2$ with high probability.
\end{remark}

\subsection{Minimax lower bound}
\label{sec2.2}

Consider the problem of constructing an estimator of $\mathbb{E}T$ under
the spectral norm based on $T$. We show that the high probability bound
in Theorem~\ref{thm:lei} is optimal up to the logarithm term in the minimax
sense.

\begin{theorem}%
\label{thm:minimax}
Suppose we observe a Bernoulli random tensor $T$ with independent entries
and $\mathbb{E}T = \theta $ for $\theta \in [0,p]^{n^{k}}$, where
$p\in (0,1]$ and $n\ge 16$. Then there exists constants
$c_{1},c_{2}>0$ only depending on $k$ such that
\begin{align*}
\inf _{\hat{\theta }}\sup _{\theta \in [0,p]^{n^{k}}}\mathbb{P}\left (\|
\hat{\theta }- \theta \|\ge (c_{1}\sqrt{np})\wedge (c_{2} n^{k/2}p)
\right )\ge \frac{1}{3},
\end{align*}
where the infimum is taken over all functions
$\hat{\theta }: \mathbb{R}^{n^{k}}\to \mathbb{R}^{n^{k}}$,
$T\mapsto \hat{\theta }(T)$. In particular, if
$p\ge \frac{c\log n}{n}$, then there exists constant $c_{3}>0$ only depending
on $k$ and $c$ such that
\begin{align*}
\inf _{\hat{\theta }}\sup _{\theta \in [0,p]^{n^{k}}}\mathbb{P}\left (\|
\hat{\theta }- \theta \|\ge c_{3}\sqrt{np} \right )\ge \frac{1}{3}.
\end{align*}
\end{theorem}

This theorem implies that in Theorem~\ref{thm:lei},
$\sqrt{np}\log ^{k-2}(n)$ cannot be replaced by other terms with order
$o(\sqrt{np})$ at least when all the entries of $A$ are one. Hence, the
upper bound is tight when $k=2$ and tight up to a logarithm factor when
$k>2$. More generally, even if we consider all functions
$\hat{\theta }: \mathbb{R}^{n^{k}}\to \mathbb{R}^{n^{k}}$,
$T\mapsto \hat{\theta }(T)$, $\|\hat{\theta }(T)-\mathbb{E}T\|$ has no high
probability bound tighter than $O(\sqrt{np})$. Therefore it is stronger
than Theorem~\ref{thm:lowerboundSpec}.

\subsection{Regularization}
\label{sec2.3}

Regularization of random graphs was first studied in
\cite{feige2005spectral}. It was proved in \cite{feige2005spectral} that
by removing high-degree vertices from a random graph, one can recover the
concentration under the spectral norm because the $O(\sqrt{np})$ concentration
breaks down when $np=O(1)$, due to the appearance of high-degree vertices,
see
\cite{krivelevich2003largest,le2017concentration,benaych2019largest}. A
data-driven threshold for finding high degree vertices for the stochastic
block model can be found in~\cite{zhou2019analysis}. A different regularization
analysis was given in \cite{le2017concentration} by decomposing the adjacency
matrix into several parts and modify a small submatrix. This method was
later generalized to other random matrices in
\cite{rebrova2018norms,Rebrova2019}. In this section we consider the regularization
for Bernoulli random tensors.

Let
$d_{i_{1},\dots ,i_{k-1}}:=\sum _{i_{k}\in [n]} t_{i_{1},i_{2},\dots ,
i_{k}}$ be the degree of the tuple $(i_{1},\dots , i_{k-1})$. When $p=\frac{c\log n}{n}$, the maximal degree
of all possible $(i_{1},\dots , i_{k-1})$ tuples is $O(np)$, and the bounded
degree property of the random hypergraph holds (see Lemma~\ref{lem:bounddeg}). When $p=o\left (\frac{\log n}{n}\right )$, the maximal
degree of all possible $(k-1)$-tuples is no longer $O(np)$, and our proof
techniques of Theorem~\ref{thm:lei} will fail in this regime without the
bounded degree property. In fact, when $p=\frac{c}{n}$ and
$\mathbb{E}T=pJ$, define a matrix $A\in \mathbb{R}^{n\times n}$ such that
$a_{i_{1},i_{2}}=t_{i_{1},i_{2},1,\dots ,1}-p$. According to
\cite{feige2005spectral}, we have
$\|T-\mathbb{E}T\|\geq \| A-\mathbb{E}A \|=\omega (\sqrt{np})$. We can see
that $\|T-\mathbb{E}T\|$ can not be bounded by $O(\sqrt{np})$ with high
probability in this regime, and the high $(k-1)$-order degrees destroy
the concentration. In general, define
%
\begin{align}
\label{eq:tupledegree}
d_{i_{1},\dots ,i_{k-m}}:=\sum _{i_{k-m+1},\dots , i_{k}\in [n]} t_{i_{1},i_{2},
\dots , i_{k}}
\end{align}
to be the degree of the tuple $(i_{1},\dots , i_{k-m})$, when
$ p =o\left (\frac{\log n}{n^{m}}\right )$ with $1\leq m<k/2$, the degrees
of all $(k-m)$-tuples fail to concentrate and our proof for~\eqref{eq:unfoldeq2} will not work in this regime. We conjecture that removing
the $(k-m)$-tuples could be a possible way to recover the concentration
of the spectral norm, but the proof of this conjecture is beyond the scope
of this paper. One obstacle is that too many tuples are removed, which
creates a large error in $\|\mathbb{E}T -\mathbb{E}\hat{T}\|$, where
$\hat{T}$ is the tensor after removing tuples of high degrees. Another
obstacle is the probability estimate. similar to
\cite{feige2005spectral,le2017concentration}, we need to take a union bound
to estimate failure probability over $2^{n^{k-m}}$ many possible sets of
removed tuples. With the union bound argument, we fail to bound the spectral
norm with high probability.

When $p\geq \frac{c\log n}{n^{m}}$ for $k/2\leq m\leq k-1$, we can still
analyze the regularization procedure by using tensor unfolding. In the
proof of Theorem~\ref{thm:main}, we have obtained
$\|T-\mathbb{E}T\|=O(\sqrt{n^{m}p})$ with high probability. The main proof
idea is that the spectral norm of $(T-\mathbb{E}T)$ can be bounded by the
spectral norm of an $n^{m}\times n^{m}$ matrix representation of
$(T-\mathbb{E}T)$, which we denote it $M$ for now. When
$p= o\left ( \frac{\log n}{n^{m}}\right )$, this tensor unfolding argument
would fail because the random matrix $M$ does not have spectral norm
$O(\sqrt{n^{m}p})$. If we regularize this matrix $M$ by removing high degree
vertices, it will still provide an upper bound on the spectral norm for
$T-\mathbb{E}T$. This is a sufficient way to obtain a spectral norm bound
when $p\geq \frac{c}{n^{m}}$. In terms of the tensor structure, before
regularization, the $(k-m)$-order degrees fail to concentrate. But after
regularization, each degree $\hat{d}_{i_{1},\dots , i_{k-m}}$ in the regularized
tensor $\hat{T}$ is bounded by $2\sqrt{n^{m} p}$, which guarantees that
the unfolded tensor is concentrated under the spectral norm.

We adapt the techniques from
\cite{feige2005spectral,le2017concentration}, together with the tensor
unfolding operation, and apply it to an inhomogeneous random directed hypergraph,
whose adjacency tensor has independent entries. This allows us to generalize
the concentration inequalities in
\cite{feige2005spectral,le2017concentration} for regularized inhomogeneous
random directed graphs with the same probability estimate. Based on different
ranges of sparsity, our regularization procedures are slightly different,
which depend on the boundedness property for different orders of degrees.

We now introduce some definitions for hypergraphs before stating the regularization
procedure.

\begin{definition}[Hypergraph]%
\label{def:hypergraph}
A \textit{hypergraph} $H$ consists of a set $V$ of \textit{vertices} and
a set $E$ of \textit{hyperedges} such that each hyperedge is a nonempty
set of $V$. $H$ is \textit{$k$-uniform} if every hyperedge $e\in E$ contains
exactly $k$ vertices. The \textit{degree} of a vertex $i$ is the number
of all hyperedges incident to $i$.
\end{definition}

Let us index the vertices by $V=\{1,\dots , n\}$. A $k$-uniform hypergraph
can be represented by order-$k$ tensor with dimension
$n\times \dots \times n$.

\begin{definition}[Adjacency tensor]%
\label{def:adj:tensor}
Given a $k$-uniform hypergraph $H$, an order-$k$ tensor $T$ is the
\textit{adjacency tensor} of $H=(V,E)$ if
\begin{align*}
t_{i_{1}, \dots , i_{k}} =
\begin{cases}
1, & \text{ if } \{i_{1},\dots , i_{k}\}\in E,
\\
0, &\text{ otherwise. }
\end{cases}
\end{align*}
For any adjacency tensor $T$,
$t_{i_{\sigma (1)},\dots ,i_{\sigma (k)}} = t_{i_{1},\dots ,i_{k}}$ for
any permutation $\sigma $. We may abuse notation and write $t_{e}$ in place
of $t_{i_{1},\dots ,i_{k}}$, where $e = \{i_{1},\dots ,i_{k}\}$. A
\textit{$k$-uniform directed hypergraph} $H=(V,E)$ consists of a set
$V$ of \textit{vertices} and a set $E$ of \textit{directed hyperedges} such
that each directed hyperedge is an element in
$ V\times \cdots \times V=V^{k}$. Define $T$ to be the adjacency tensor of
$H$ such that
\begin{align*}
t_{i_{1}, \dots , i_{k}} =
\begin{cases}
1, & \text{ if } (i_{1},\dots ,i_{k})\in E,
\\
0, &\text{ otherwise. }
\end{cases}
\end{align*}
\end{definition}

For different orders of sparsity in terms of $m$, our regularization procedure
is based on the boundedness of $(k-m)$-th order degrees in the random directed
hypergraph. Assume $p\geq \frac{c}{n^{m}}$ with $k/2\leq m\leq k-1$. For
any order-$k$ tensor $A$ indexed by $[n]$, let $S\subset [n]^{k-m}$ be
a subset of indices. We define the regularized tensor $A^{S}$ as
\begin{align*}
a_{i_{1},\dots , i_{k}}^{S}=
\begin{cases}
0 & \text{if } (i_{1},\dots ,i_{k-m})\in S,
\\
a_{i_{1},\dots , i_{k}} & \text{otherwise}.
\end{cases}
\end{align*}

When we observe a random tensor $T$, we regularize $T$ as follows. Suppose
the degree of a $(k-m)$-tuple $(i_{1},\dots ,i_{k-m})$ is greater than
$2n^{m}p$, then we remove all directed hyperedges containing this tuple.
In other words, we zero out the corresponding hyperedges in the adjacency
tensor. We call the resulting tensor $\hat{T}$. Let
$\tilde{S}\subset [n]^{k-m}$ be the set of $(k-m)$-tuples with degree greater
than $2n^{m}p$. Then with our notation,
%
\begin{align}
\label{eq:regularizationdef}
\hat{T}=T^{\tilde{S}}.
\end{align}

Since from our Theorem~\ref{thm:main}, when
$p\geq \frac{c\log n}{n^{m}}$ for some $c>0$, the regularization is not
needed, below we assume $p<\frac{\log n}{n^{m}}$ for simplicity. The following
lemma shows that with high probability, not many $(k-m)$-tuples are removed.

\begin{lemma}%
\label{lem:num:regularized}
Let $\frac{c}{n^{m}}\le p<\frac{\log n}{n^{m}}$ for a sufficiently large
$c>1$ and an integer $m$ with $k/2 \leq m\leq k-1$. Then the number of
regularized $(k-m)$-tuples $|\tilde{S}|$ is at most
$\frac{1}{n^{2m-k}p}$ with probability at least
$1-\exp \left (-\frac{n^{k-m}}{6\log n}\right )$.
\end{lemma}

After regularization, the following holds.
%
\begin{theorem}%
\label{thm:regularize}
Let $\frac{c}{n^{m}}\le p<\frac{\log n}{n^{m}}$ for a sufficiently large
$c>1$ and an integer $m$ with $k/2 \leq m\leq k-1$. Let $\hat{T}$ be the
random order-$k$ tensor $T$ after regularization. Then for any $r>0$, there
exists a constant $C_{2}$ depending on $k,r$ such that
%
\begin{align}
\label{eq:regularization2}
\mathbb{P}\left ( \|\hat{T} - \mathbb{E} T\| \le C_{2}\sqrt{n ^{m}p}
\right )\geq 1-n^{-r}.
\end{align}
\end{theorem}

\section{Applications}
\label{sec3}

To demonstrate the usefulness of our concentration and regularization results,
we highlight two applications.

\subsection{Sparse hypergraph expanders}
\label{sec:hypergraphexpander}

The expander mixing lemma for a $d$-regular graph (the degree of each vertex
is $d$) states the following: Let $G$ be a $d$-regular graph on $n$ vertices
with the second largest eigenvalue in absolute value of its adjacency matrix
satisfying $\lambda :=\max \{\lambda _{2},|\lambda _{n}|\}<d$. For any
two subsets $V_{1}, V_{2} \subseteq V(G)$, let
\begin{equation*}
e(V_{1},V_{2})=|\{(v_{1},v_{2})\in V_{1}\times V_{2}: \{v_{1},v_{2}\}
\in E(G) \}|
\end{equation*}
be the number of edges between $V_{1}$ and $V_{2}$. Then
%
\begin{align}
\label{eq:expandermixing}
\left | e(V_{1},V_{2})-\frac{d|V_{1}||V_{2}|}{n}\right | \leq
\lambda \sqrt{|V_{1}| |V_{2}|\left (1-\frac{|V_{1}|}{n}\right )
\left (1-\frac{|V_{2}|}{n}\right )}.
\end{align}\eqref{eq:expandermixing} shows that $d$-regular graphs with small
$\lambda $ have a good mixing property, where the number of edges between
any two vertex subsets is approximated by the number of edges we would
expect if they were drawn at random. Such graphs are called expanders,
and the quality of such an approximation is controlled by $\lambda $, which
is also the mixing rate of simple random walks on $G$
\cite{chungspectral}.

Hypergraph expanders have recently received significant attention in combinatorics
and theoretical computer science
\cite{lubotzky2017high,dinur2017high}. Many different definitions have
been proposed for hypergraph expanders, each with their own strengths and
weaknesses. In this section, we only consider hypergraph models that have
a generalized version of expander mixing lemma~\eqref{eq:expandermixing}.

There are several hypergraph expander mixing lemmas in the literature based
on the spectral norm of tensors
\cite{friedman1995second,lenz2015eigenvalues,cohen2016inverse}. However,
for deterministic tensors, the spectral norm is NP-hard to compute
\cite{hillar2013most}, hence those estimates might not be applicable in
practice. In \cite{bilu2004codes,harris2019deterministic}, the authors
obtained a weaker expander version mixing lemma for a sparse deterministic
hypergraph model where the mixing property depends on the second eigenvalue
of a regular graph. Friedman and Widgerson \cite{friedman1995second} obtained
the following spectral norm bound for a random hypergraph model: Consider
a $k$-uniform hypergraph model on $n$ vertices where $ n^{k}p$ hyperedges
are chosen independently at random. Let $J$ be the order-$k$ tensor with
all entries taking value $1$. If $p\geq Ck\log n/n$, then with high probability
$ \left \|  T-pJ\right \|  \leq (C\log n)^{k/2}\sqrt{np}$. Combining with
their expander mixing lemma in \cite{friedman1995second}, it provides a
random hypergraph model with a good control of the mixing property. This
is a random hypergraph model with expected degrees ${n^{k-1}p}$, which
is not bounded. To the best of our knowledge, our Theorem~\ref{thm:mixing} below is the first construction of a sparse random hypergraph
model with bounded degrees that satisfies a $k$-subset expander mixing
lemma with high probability. The idea of applying expander mixing lemma
and spectral gap results of sparse expanders to analyze matrix completion
and tensor completion has been developed in
\cite{heiman2014deterministic,bhojanapalli2014universal,brito2018spectral,gamarnik2017matrix,harris2019deterministic}.
We believe our result could also be useful for tensor completion or other
related problems.

Let $H$ be a $k$-uniform Erd\H{o}s-R\'{e}nyi hypergraph (recall Definition~\ref{def:hypergraph}) on $n$ vertices with sparsity
$p=\frac{c}{n^{k-1}}$, where each hyperedge is generated independently
with probability $p$. Its adjacency tensor is then a symmetric tensor,
denoted by~$T$. We construct a regularized hypergraph $H'$ as follows:
\begin{enumerate}
\item Construct $\tilde{T}$ such that
%
\begin{align}
\label{eq:symmetricregularize}
\tilde{t}_{i_{1},\dots ,i_{k}}=
\begin{cases}
t_{i_{1},\dots ,i_{k}} & \text{if } 1\leq i_{1}<i_{2}<\dots < i_{k}
\leq n,
\\
0 & \text{otherwise.}
\end{cases}
\end{align}
\label{eq:step1}
\item Compute
$ \tilde{d}_{i}:=\sum _{i_{1},\dots , i_{k-1}\in [n]} \tilde{t}_{i,i_{1},
\dots ,i_{k-1}}$ for all $i\in [n]$. If $\tilde{d}_{i}>2n^{k-1}p$, zero
out all entries
$\tilde{t}_{i,i_{1},\dots ,i_{k-1}}, i_{1},\dots , i_{k-1} \in [n]$. We
then obtain a new tensor $\hat{T}$.
\item Define $T'$ such that
$(t')_{i_{1},i_{2},\dots ,i_{k}}=\sum _{\sigma \in \mathfrak{S}_{k}}
\hat{t}_{i_{\sigma (1)},\dots , i_{\sigma (k)}}$, where
$\mathfrak{S}_{k}$ is the symmetric group of order $k$. We then obtain a
regularized hypergraph $H'$ with adjacency tensor $T'$.%
\label{eq:step3}
\end{enumerate}

Note that this regularization procedure is applicable to inhomogeneous
random hypergraphs by taking
$ p=\max _{i_{1},\dots ,i_{k}\in [n]} p_{i_{1},\dots ,i_{k}}$. By our construction,
$H'$ is a $k$-uniform hypergraph with degrees at most
$2k!n^{k-1}p=2k!c$. Let $J\in \mathbb{R}^{n^{k}}$ be an order $k$ tensor
with all entries taking value $1$. From Theorem~\ref{thm:regularize}, for
some constant $C>0$, with high probability its adjacency tensor $T'$ satisfies
%
\begin{align}
\label{eq:expandreg}
\|T'-p J\| \leq C\sqrt{n^{k-1}p}.
\end{align}
In the next theorem we use~\eqref{eq:expandreg} to show that $H'$ satisfies
an expander mixing lemma with high probability.

\begin{definition}
If $V_{1},\dots , V_{k}$ are subsets of $V(H)$ for a $k$-uniform hypergraph
$H$, define
%
\begin{align}
e_{H}(V_{1},\dots V_{k}):=|\{(v_{1},\dots , v_{k})\in V_{1}\times
\cdots \times V_{k}: \{v_{1},\dots , v_{k}\}\in E(H) \}|
\end{align}
to be the number of hyperedges between $V_{1},\dots , V_{k}$.
\end{definition}

\begin{theorem}%
\label{thm:mixing}
Let $H$ be a $k$-uniform Erd\H{o}s-R\'{e}nyi hypergraph on $n$ vertices
with sparsity $p= \frac{c}{n^{k-1}}$ for some sufficiently large constant
$c>1$. Let $H'$ be the hypergraph $H$ after regularization, then there
exists a constant $C>0$ depending on $k$ such that with high probability
for any non-empty subsets $V_{1},\dots , V_{k}\subset [n]$, we have the
following expander mixing inequality:
%
\begin{align}
\label{eq:mixinglemma}
\left | e_{H'}(V_{1},\dots , V_{k})-p|V_{1}|\cdots |V_{k}|\right |
\leq C\sqrt{c} \cdot \sqrt{|V_{1}|\cdots |V_{k}|}.
\end{align}
\end{theorem}

\subsection{Tensor sparsification}
\label{sec:sparsification}

In the \textit{tensor completion} problem, one aims to estimate a low-rank
tensor based on a random sample of observed entries. A commonly used
definition of the rank for tensors is called canonical polyadic (CP) rank. We
refer to \cite{kolda2009tensor} for more details. In order to solve a tensor
completion problems, there are two main steps. First, one needs to sample
some entries from a low-rank tensor $T$. Then, based on the observed entries,
one solves an optimization problem and justifies that the solutions to this
problem will be exactly or nearly the original tensor~$T$. A fundamental
question is: how many observations are required to guarantee that the
solution of the optimization problem provides a good recovery of the original
tensor $T$?

After a random sampling from the original tensor $T$, we obtain a random
tensor $\tilde{T}$. If we require the sample size to be small,
$\tilde{T}$ then will be random and sparse. In the next step, the optimization
procedure is then based on $\tilde{T}$. In the matrix or tensor completion
algorithm, especially for the non-convex optimization algorithm, we need
some stability guarantee on the initial data, see for example
\cite{keshavan2010matrix,jain2014provable,cai2019}. Therefore, it is important
to have concentration guarantee such that $\tilde{T}$ is close to
$T$ under the spectral norm.

Another related problem is called \textit{tensor sparsification}. Given
a tensor $T$, through some sampling algorithm, one wants to construct a
sparsified version $\tilde{T}$ of $T$ such that $\|T-\tilde{T}\|$ is relatively
small with high probability. In \cite{nguyen2015tensor}, a non-uniform
sampling algorithm was proposed and the probability of sampling each entry
is chosen based on the magnitude of the entry in $T$. However, without
knowing the exact value of the original tensor $T$, a reasonable way to
output a sparsified tensor $\tilde{T}$ is through uniform sampling.

We obtain a concentration inequality of the spectral norm for tensors under
uniform sampling, which is useful to both of the problems above. It improves
the sparsity assumption in the analysis of the initialization step for
the tensor completion algorithm proposed in \cite{jain2014provable}. Let
$T$ be a deterministic tensor. We obtain a new tensor $\tilde{T}$ by uniformly
sampling entries in $T$ with probability $p$. Namely,
\begin{align*}
\tilde{t}_{i_{1},\dots ,i_{k}}=
\begin{cases}
t_{i_{1},\dots ,i_{k}} & \text{with probability } p,
\\
0 & \text{with probability } 1-p.
\end{cases}
\end{align*}
By our definition, $\mathbb{E}\tilde{T}=pT$. The following is an estimate
about the concentration of $\tilde{T}$ under the spectral norm when
$p\geq \frac{c\log n}{n^{m}}$, $1\leq m\leq k-1$. It is a quick corollary
from Theorem~\ref{thm:lei} and Theorem~\ref{thm:main}.
%
\begin{corollary}%
\label{thm:sampling}
Let $p\geq \frac{c\log n}{n^{m}}$ for some constant $c>0$ and some integer
$1\leq m\leq k-1$. Denote
$t_{\max }:=\max _{i_{1},\dots ,i_{k}\in [n]}|t_{i_{1},\dots ,i_{k}}|$.
For any $r>0$, there exists a constant $C>0$ depending on $r,k,c$ such
that with probability $1-n^{-r}$,
\begin{align*}
\|\tilde{T}-pT\| \leq
\begin{cases}
Ct_{\max }\sqrt{n^{m}p} & \text{if} \quad k/2\leq m\leq k-1,
\\
Ct_{\max } \sqrt{n^{m}p} \log ^{\frac{k-1}{m}-1}(n) & \text{if} \quad 1
\leq m<k/2, \quad k/m\not \in \mathbb{Z},
\\
Ct_{\max }\sqrt{n^{m}p} \log ^{\frac{k}{m}-2}(n) & \text{if} \quad 1
\leq m <k/2,\quad k/m \in \mathbb{Z}.
\end{cases}
\end{align*}
\end{corollary}
%
\begin{remark}
Theorem 2.1 in \cite{jain2014provable} provides an estimate of
$\|\tilde{T}-pT\|$ for a symmetric tensor $T$ with symmetric sampling,
assuming $k=3$ and $p\geq \frac{\log n}{n^{3/2}}$. When $k=3$, our result
allows the sparsity down to $p\geq \frac{c\log n}{n^{2}}$ and covers non-symmetric
tensors with uniform sampling.
\end{remark}

\section{Proof of Theorem \ref{thm:lei}}
\label{sec:prooflei}

The proof is a generalization of
\cite{feige2005spectral,lei2015consistency} and is suitable for sparse
random tensors. This type of method is known as the Kahn-Szemer\'{e}di
argument originally introduced in \cite{friedman1989}. Without loss of
generality, we may assume
%
\begin{align}
\label{eq:max1}
\max _{i_{1},\dots ,i_{k}}|A_{i_{1},\dots ,i_{k}}|=1
\end{align}
in our proof for simplicity.

\subsection{Discretization
}
\label{sec:discretization}

Fix $\delta \in (0,1)$, define the $n$-dimensional ball of radius
$t$ as
\begin{equation*}
S_{t}:=\{v\in \mathbb{R}^{n}: \|v\|_{2}\leq t \}.
\end{equation*}
We introduce a set of lattice points in $ S_{1}$ as follows:
%
\begin{align}
\label{eq:lattice}
\mathcal{T}=\left \{  {x}=(x_{1},\dots x_{n})\in S_{1}:
\frac{\sqrt{n}x_{i}}{\delta } \in \mathbb{Z},\forall i\in [n]\right \}  .
\end{align}

By the Lipschitz property of spectral norms, we have the following upper
bound, which reduces the problem of bounding the spectral norm of
$T$ to an optimization problem over $\mathcal{T}$.
%
\begin{lemma}%
\label{lem:convex}
For any tensor $T\in \mathbb{R}^{n^{k}}$ and any fixed
$\delta \in (0,1)$, we have
\begin{align*}
\|T\|\leq (1-\delta )^{-k}\sup _{y_{1},\dots , y_{k}\in \mathcal{T}} |T(y_{1},
\dots , y_{k})|.
\end{align*}
\end{lemma}
\begin{proof}
The proof follows from Lemma 2.1 in the supplement of
\cite{lei2015consistency}. For completeness, we provide the proof here.
For any $v\in S_{1-\delta }$, consider the cube in $\mathbb{R}^{n}$ of edge
length $\delta /\sqrt{n}$ that contains $v$, with all its vertices in
$\left (\frac{\delta }{\sqrt{n}}\mathbb{Z}\right )^{n}$. The diameter of the
cube is $\delta $, so the entire cube is contained in $S_{1}$. Hence all
vertices of this cube are in $\mathcal{T}$ and
$S_{1-\delta }\subset \textnormal{convhull}(\mathcal{T})$. Therefore for
each $u_{i}\in S_{1}, 1\leq i\leq k$, we can find some sequence
$\{x_{i_{j}}\}_{j=1}^{N_{i}}\subset \mathcal{T}$ such that
$(1-\delta )u_{i}$ is a linear combination of those $\{x_{i_{j}}\}$, namely,
\begin{equation*}
(1-\delta )u_{i}=\sum _{j=1}^{N_{i}} a_{j}^{(i)} {x_{i_{j}}},
\end{equation*}
for some $a_{j}^{(i)}\in [0,1]$ satisfying
$ \sum _{j=1}^{N_{i}} a_{j}^{(i)}=1$. Then
\begin{align*}
|T(u_{1},\dots ,u_{k})| &=(1-\delta )^{-k} |T((1-\delta )u_{1},\dots ,(1-
\delta )u_{k})|
\\
&\leq (1-\delta )^{-k}\sum _{j_{1}=1}^{N_{1}} \dots \sum _{j_{k}=1 }^{N_{k}}
a_{j_{1}}^{(1)}\cdots a_{j_{k}}^{(k)}|T( {x_{1_{j_{1}}}},\dots , {x_{k_{j_{k}}}})|
\\
&\leq (1-\delta )^{-k}\sup _{y_{1},\dots , y_{k}\in \mathcal{T}}|T(y_{1},
\dots ,y_{k})|,
\end{align*}
where the last inequality is due to
\begin{align*}
\sum _{j_{1}=1}^{N_{1}} \dots \sum _{j_{k}=1}^{N_{k}}a_{j_{1}}^{(1)}
\cdots a_{j_{k}}^{(k)} = \prod _{i=1}^{k} \left (\sum _{j_{i}=1}^{N_{i}}
a_{j_{i}}^{(i)}\right )=1.
\end{align*}
This completes the proof.
\end{proof}

Now for any fixed $k$-tuples
$(y_{1},\dots , y_{k})\in \mathcal{T}\times \cdots \times \mathcal{T}$, we
decompose its index set. Define the set of \textit{light tuples} as
%
\begin{align}
\label{eq:deflight}
\mathcal{L}=\mathcal{L}(y_{1},\dots , y_{k}):=\left \{  (i_{1},\dots ,i_{k})
\in [n]^{k}: |y_{1, i_{1}}\cdots y_{k, i_{k}}|\leq
\frac{\sqrt{np}}{n}\right \}  ,
\end{align}
and \textit{heavy tuples} as
%
\begin{align}
\label{eq:defheavy}
\overline{\mathcal{L}}=\overline{\mathcal{L}}(y_{1},\dots , y_{k}):=
\left \{  (i_{1},\dots ,i_{k})\in [n]^{k}: |y_{1,i_{1}}\cdots y_{k,i_{k}}|
> \frac{\sqrt{np}}{n}\right \}  .
\end{align}

In the remaining part of our proof, we control the contributions of light
and heavy tuples to the spectral norm respectively.

\subsection{Light tuples}
\label{sec4.2}

Let $W:=A\circ T-\mathbb{E}[A\circ T]$ be the centered random tensor and
we denote the entries of $W$ by $w_{i_{1},\dots ,i_{k}}$ for
$i_{1},\dots , i_{k}\in [n]$. We have the following concentration bound
for the contribution of light tuples.

\begin{lemma}%
\label{lem:light}
For any constant $c>0$,
\begin{align*}
&\mathbb{P}\left (\sup _{y_{1},\dots ,y_{k}\in \mathcal{T}}\left |\sum _{(i_{1},
\dots , i_{k})\in \mathcal{L} } y_{1,i_{1}}\cdots y_{k,i_{k}}w_{i_{1},
\dots , i_{k}} \right |\geq c\sqrt{np} \right )
\\
&\leq 2\exp \left [ -n\left (\frac{c^{2}}{2(1+c/3)}-k\log \Big (
\frac{7}{\delta }\Big )\right ) \right ],
\end{align*}
where $\delta \in (0,1)$ on the right hand side is determined by the definition
of $\mathcal{T}$ in~\eqref{eq:lattice}.
\end{lemma}
\begin{proof}
Denote
%
\begin{align}
\label{eq:u}
u_{i_{1},\dots , i_{k}}:=y_{1,i_{1}}\cdots y_{k,i_{k}}\mathbf{1}\{|y_{1,i_{1}}
\cdots y_{k,i_{k}}|\leq \sqrt{np}/n\}.
\end{align}
Then the contribution from light tuples can be written as
\begin{equation*}
\sum _{ i_{1},\dots ,i_{k}\in [n]} w_{i_{1},\dots , i_{k}} u_{i_{1},
\dots , i_{k}}.
\end{equation*}
Since from~\eqref{eq:max1}, each term in the sum has mean zero and is bounded
by $\sqrt{np}/n$, we are ready to apply Bernstein's inequality to get for
any constant $c>0$,
%
\begin{align}
&\mathbb{P}\left ( \left |\sum _{ i_{1},\dots ,i_{k}\in [n]} w_{i_{1},
\dots , i_{k}} u_{i_{1},\dots , i_{k}}\right |\geq c\sqrt{np}\right )
\\
\leq & 2 \exp \left ( -
\frac{c^{2} np/2}{\sum _{i_{1}, \dots ,i_{k}\in [n]} p_{i_{1},\dots , i_{k}}(1-p_{i_{1},\dots ,i_{k}})u_{i_{1},\dots , i_{k}}^{2}+\frac{1}{3}\frac{\sqrt{np}}{n} c\sqrt{np}}
\right )
\notag
\\
\leq & 2 \exp \left ( -
\frac{c^{2} np/2}{p\sum _{i_{1}, \dots ,i_{k}\in [n]} u_{i_{1},\dots , i_{k}}^{2}+ \frac{cp}{3} }
\right ).
\label{eq:tail}
\end{align}
From~\eqref{eq:u} we have
\begin{align*}
\sum _{i_{1}, \dots ,i_{k}\in [n]} u_{i_{1},\dots , i_{k}}^{2} &\leq
\sum _{i_{1},\dots , i_{k}\in [n]}y_{1,i_{1}}^{2}\cdots y_{k,i_{k}}^{2}
= \prod _{j=1}^{k} \|y_{j}\|_{2}^{2}=1.
\end{align*}
Then~\eqref{eq:tail} is bounded by
$ 2\exp \left ( \frac{-c^{2}n}{2+\frac{2c}{3}}\right ) $. By the volume
argument (see for example \cite{vershynin2018high}) we have
$|\mathcal{T}|\leq \exp (n\log (7/\delta ))$, hence the $k$-th product of
$\mathcal{T}$ satisfies
\begin{equation*}
|\mathcal{T}\times \cdots \mathcal{\times }\mathcal{T}|\leq \exp (kn
\log (7/\delta )).
\end{equation*}
Then taking a union bound over all possible
$y_{1},\dots , y_{k}\in \mathcal{T}$, we have
\begin{align*}
\sup _{y_{1},\dots , y_{k}\in \mathcal{T}}\left |\sum _{(i_{1},\dots ,i_{k})
\in \mathcal{L} } y_{1,i_{1}}\cdots y_{k,i_{k}}w_{i_{1},\dots , i_{k}}
\right |\leq c\sqrt{np}
\end{align*}
with probability at least
$ 1-2\exp \left [ -\frac{c^{2}n}{2(1+c/3)}+kn\log (7/\delta )\right ]
$. This completes the proof.
\end{proof}

By Lemma~\ref{lem:light}, for any $r>0$, we can take the constant
$c$ in Lemma~\ref{lem:light} large enough depending on $k$ and $r$ (for
example, taking $c=6r+6k\log (7/\delta )$ suffices) such that with probability
at least $1-n^{-r}$,
\begin{equation*}
\sup _{y_{1},\dots ,y_{k}\in \mathcal{T}}\left |\sum _{i_{1},\dots , i_{k}
\in \mathcal{L} } y_{1,i_{1}}\cdots y_{k,i_{k}}w_{i_{1},\dots ,i_{k}}
\right |\leq c\sqrt{np}.
\end{equation*}
Therefore to prove Theorem~\ref{thm:main}, it remains to control the contribution
from heavy tuples.

\subsection{Heavy tuples}
\label{sec:heavy}

Next, we show the contribution from heavy tuples is bounded by
$c\sqrt{np}\log ^{k-2}(n)$ for some constant $c>0$ depending on $k$ with
high enough probability. Namely,
\begin{equation*}
\sup _{y_{1},\dots , y_{k}\in \mathcal{T}}\left | \sum _{(i_{1},\dots ,i_{k})
\in \overline{\mathcal{L}}} y_{1,i_{1}}\cdots y_{k,i_{k}}\cdot w_{i_{1},
\dots ,i_{k}}\right |\leq c\sqrt{np}\log ^{k-2}(n).
\end{equation*}
Note that from our definition of heavy tuples in~\eqref{eq:defheavy}, we
have
%
\begin{align}
&\left | \sum _{(i_{1},\dots ,i_{k})\in \overline{\mathcal{L}} } y_{1,i_{1}}
\cdots y_{k,i_{k}}\cdot (a_{i_{1},\dots ,i_{t}}p_{i_{1},\dots ,i_{k}})
\right |
\notag
\\
\leq & \sum _{(i_{1},\dots ,i_{k})\in \overline{\mathcal{L}}}
\frac{y_{1,i_{1}}^{2}\cdots y_{k,i_{k}}^{2}}{|y_{1,i_{1}}\cdots y_{k,i_{k}}|}
\cdot p_{i_{1},\dots ,i_{k}} \leq \sum _{(i_{1},\dots ,i_{k})\in
\overline{\mathcal{L}}} \frac{n}{\sqrt{np}}{y_{1,i_{1}}^{2}\cdots y_{k,i_{k}}^{2}}
\cdot p
\notag
\\
\leq & \sqrt{np}\sum _{(i_{1},\dots ,i_{k})\in \overline{\mathcal{L}}} {y_{1,i_{1}}^{2}
\cdots y_{k,i_{k}}^{2}} \leq \sqrt{np}.
\label{eq:heavyboundforp}
\end{align}
Then it suffices to show that with high enough probability for all
$y_{1},\dots ,y_{k}\in \mathcal{T}$,
%
\begin{align}
\label{eq:heavy}
\left | \sum _{(i_{1},\dots ,i_{k})\in \overline{\mathcal{L}}} y_{1,i_{1}}
\cdots y_{k,i_{k}}\cdot (a_{i_{1},\dots ,i_{k}}t_{i_{1},\dots ,i_{k}})
\right |\leq C_{k}\sqrt{np}\log ^{k-2}(n)
\end{align}
for a constant $C_{k}$ depending on $k$. We will focus on the heavy tuples
$(i_{1},\dots , i_{k})$ such that
$y_{1,i_{1}},\dots , y_{k,i_{k}}>0$. We denote this set by
$\overline{\mathcal{L}}^{+}$. The remaining cases will be similar and there
are $2^{k}$ different cases in total. Note that
%
\begin{align}
\label{eq:reduction}
&\left | \sum _{(i_{1},\dots ,i_{k})\in \overline{\mathcal{L}}^{+}} y_{1,i_{1}}
\cdots y_{k,i_{k}}\cdot {(a_{i_{1},\dots ,i_{k}}t_{i_{1},\dots ,i_{k}})}
\right |
\notag
\\
\leq & \sum _{(i_{1},\dots ,i_{k})\in \overline{\mathcal{L}}^{+}} y_{1,i_{1}}
\cdots y_{k,i_{k}}\cdot {t_{i_{1},\dots ,i_{k}}}.
\end{align}
For the rest of the proof we will bound the right hand side of~\eqref{eq:reduction}. We now define the following index sets for a fixed
tuple
$(y_{1},\dots , y_{k})\in \mathcal{T}\times \cdots \times \mathcal{T}$:
%
\begin{align}
\label{eq:rangeS}
D_{j}^{s}=\left \{  i: \frac{2^{s-1}\delta }{\sqrt{n}}\leq y_{j,i}\leq
\frac{2^{s}\delta }{\sqrt{n}} \right \}  \text{ for } s=1,\dots , \left
\lceil \log _{2} (\sqrt{n}/\delta ) \right \rceil \text{ and } 1\leq j
\leq k.
\end{align}
By our definition of $\mathcal{T}$ in~\eqref{eq:lattice}, for any
$(y_{1},\dots , y_{k})\in (\mathcal{T}\times \cdots \times \mathcal{T})$
and $(i_{1},\dots ,i_{k})\in \overline{\mathcal{L}}^{+}$, we have
$y_{j,i_{j}}\geq \delta /\sqrt{n}$ for all $1\leq j\leq k$. Therefore each
$i_{j}$ is in $D_{j}^{s}$ for some $s$. Recall the definition of degree
for a $(k-1)$-tuple in~\eqref{eq:tupledegree}. Also the following definitions
are needed:
\begin{enumerate}
\item
\begin{equation*}
e(I_{1},\dots , I_{k}):=\left |\{ (i_{1},\dots , i_{k}): t_{i_{1},
\dots , i_{k}}=1, i_{1}\in I_{1},\dots , i_{k}\in I_{k}\} \right |.
\end{equation*}
\item
\begin{equation*}
\mu (I_{1},\dots , I_{k})=\mathbb{E} e(I_{1},\dots , I_{k}), \quad
\overline{\mu }(I_{1},\dots , I_{k})=p|I_{1}|\cdots |I_{k}|.
\end{equation*}
\item For
$1\leq s_{1},\dots ,s_{k}\leq \log _{2} (\sqrt{n}/\delta )$,
\begin{equation*}
\lambda _{s_{1},\dots , s_{k}}=
\frac{e(D_{1}^{s_{1}},\dots ,D_{k}^{s_{k}})}{\overline{\mu }_{s_{1},\dots , s_{k}}},
\quad \overline{\mu }_{s_{1},\dots , s_{k}}=\overline{\mu }(D_{1}^{s_{1}},
\dots ,D_{k}^{s_{k}}).
\end{equation*}
\label{def:lambdamu}
\item
$\ \alpha _{j,s}=|D_{j}^{s}|\cdot 2^{2s}/n, 1\leq j\leq k, 1\leq s
\leq \log _{2} (\sqrt{n}/\delta )$.
\label{eq:defalpha}
\item
$\sigma _{s_{1},\dots , s_{k}}=\lambda _{s_{1},\dots , s_{k}}n^{k/2-1}
\sqrt{np}\cdot 2^{-(s_{1}+\cdots +s_{k})}, 1\leq s_{1},\dots ,s_{k}
\leq \log _{2} (\sqrt{n}/\delta )$.%
\label{def:alphasigma}
\end{enumerate}

The following two lemmas are about the properties of the sparse directed
random hypergraphs (recall Definition~\ref{def:adj:tensor}), which are
important for the rest of our proof.

\begin{lemma}[Bounded degree]%
\label{lem:bounddeg}
Assume $p\geq c \log n/n$ for some constant $c>0$. Then for any
$r>0$, there exists a constant $c_{1}>1$ depending on $c,r,k$ such that
with probability at least $1-n^{-r}$, for all
$i_{1},\dots ,i_{k-1}\in [n]$,
$d_{i_{1},\dots ,i_{k-1}}\leq c_{1}np $.
\end{lemma}
\begin{proof}
For a fixed $(k-1)$-tuple $(i_{1},\dots ,i_{k-1})$, by Bernstein's inequality,
%
\begin{align}
\mathbb{P}( d_{i_{1},\dots ,i_{k-1}}\ge c_{1}np ) &=\mathbb{P} \left (
\sum _{i_{k}\in [n]} t_{i_{1},\dots , i_{k}} \ge c_{1}n p \right )
\notag
\\
&\le \mathbb{P} \left ( \sum _{ i_{k}\in [n]} (t_{i_{1},\dots , i_{k}}-p_{i_{1},
\dots ,i_{k}}) \ge (c_{1}-1)n p \right )
\notag
\\
&\le \exp \left [ -
\frac{\frac{1}{2} (c_{1}-1)^{2}n^{2}p^{2}}{ n p + \frac{1}{3}(c_{1}-1)n p }
\right ] \leq n^{-\frac{3(c_{1}-1)^{2}c}{4+2c_{1}}},
\label{eq:boundedelemma}
\end{align}
where in the last inequality we use the assumption $p\geq c\log n/n$. Then
taking a union bound over $i_{1},\dots ,i_{k-1}\in [n]$ implies
\begin{align*}
\mathbb{P}\left ( \sup _{i_{1},\dots ,i_{k-1}\in [n]} d_{i_{1},\dots ,i_{k-1}}
\ge c_{1} np \right ) \leq n^{-\frac{3(c_{1}-1)^{2}c}{4+2c_{1}}+k-1}.
\end{align*}
Therefore for any $r,c>0$ we can take $c_{1}$ sufficiently large (depending
linearly on $k,r$) to make Lemma~\eqref{lem:bounddeg} hold.
\end{proof}

\begin{lemma}[Bounded discrepancy]%
\label{lem:bounddiscrepancy}
Assume $p\geq c \log n/n$ for some constant $c>0$. For any $r>0$, there
exist constants $c_{2},c_{3}>1$ depending on $c,r,k$ such that with probability
at least $1-2n^{-r}$, for any nonempty sets
$I_{1},\dots , I_{k}\subset [n]$ with
$1\leq |I_{1}|\leq \cdots \leq |I_{k}|$, at least one of the following
events hold:
\begin{enumerate}
\item
$ \frac{e(I_{1},\dots , I_{k})}{\overline{\mu }(I_{1},\dots , I_{k})}
\leq e c_{2} $,%
\label{case:1}
\item
$ e(I_{1},\dots ,I_{k})\log \left (
\frac{e(I_{1},\dots , I_{k})}{\overline{\mu }(I_{1},\dots , I_{k})}
\right )\leq c_{3} |I_{k}|\log \left (\frac{n}{|I_{k}|} \right ) $.%
\label{case:2}
\end{enumerate}
\end{lemma}
We will use the following Chernoff's inequality (see
\cite{boucheron2013concentration}).
%
\begin{lemma}[Chernoff bound]
Let $X_{1},\dots , X_{n}$ be independent Bernoulli random variables. Let
$X=\sum _{i=1}^{n} X_{n}$ and $\mu =\mathbb{E}X$. Then for any
$\delta >0$,
%
\begin{align}
\label{eq:strongcher}
\mathbb{P} (X>(1+\delta )\mu )\leq \exp (-\mu ((1+\delta )\ln (1+
\delta )-\delta )).
\end{align}
In particular, we have a weaker version of~\eqref{eq:strongcher}: for any
$\delta >0$,
%
\begin{align}
\label{eq:weakercher}
\mathbb{P} (X>(1+\delta )\mu )\leq \exp \left (
\frac{-\delta ^{2}\mu }{2+\delta } \right ) .
\end{align}
\end{lemma}

\begin{proof}[Proof of Lemma~\ref{lem:bounddiscrepancy}]
In this proof, we assume the event in Lemma~\ref{lem:bounddeg} that all
degrees of vertices are bounded by $c_{1} n p$ holds. If
$|I_{k}|\ge n/e$, then the bounded degree property implies
$e(I_{1}, \dots , I_{k})\leq |I_{1}|\cdots |I_{k-1}|c_{1}np$. Hence
\begin{align*}
\frac{e(I_{1}, \dots , I_{k})}{\overline{\mu }(I_{1},\dots ,I_{k})}=
\frac{e(I_{1}, \dots , I_{k})}{p|I_{1}|\cdots |I_{k}|} \le
\frac{|I_{1}|\cdots |I_{k-1}|c_{1}np}{p|I_{1}|\cdots |I_{k-1}|n/e}
\leq c_{1}e.
\end{align*}
This completes the proof for Case~\eqref{case:1}. Now we consider the case
when $|I_{k}|< n/e$. Let $s(I_{1},\dots , I_{k})$ be the set of all possible
distinct hyperedges between $I_{1},\dots , I_{k}$. We have for any
$\tau >1$ and any fixed $I_{1},\dots ,I_{k}$,
\begin{align*}
&\mathbb{P}(e(I_{1}, \dots , I_{k})\ge \tau \bar{\mu }(I_{1}, \dots , I_{k}))
\\
=& \mathbb{P}\left (\sum _{(i_{1},\dots ,i_{k})\in s(I_{1},\dots ,I_{k})}t_{i_{1},
\dots ,i_{k}} \ge \tau \bar{\mu }(I_{1},\dots ,I_{k}) \right )
\\
\le & \mathbb{P}\left (\sum _{(i_{1},\dots ,i_{k})\in s(I_{1},\dots ,I_{k})}
(t_{i_{1},\dots ,i_{k}}-p_{i_{1},\dots ,i_{k}})\ge (\tau -1)\bar{\mu }(I_{1},
\dots , I_{k})\right ).
\end{align*}
By Chernoff's inequality~\eqref{eq:strongcher}, the last line above is
bounded by
\begin{align*}
&\exp \left ( (\tau -1)\overline{\mu }(I_{1},\dots ,I_{k})-\tau
\overline{\mu }(I_{1},\dots ,I_{k})\log \tau \right )
\\
\leq & \exp \left ( -\frac{1}{2} \bar{\mu }(I_{1}, \dots , I_{k}) \tau
\log \tau \right ),
\end{align*}
where the last inequality holds when $\tau \geq 8$. This implies for
$\tau \geq 8$,
%
\begin{align}
\label{eq:chernoffbound}
\mathbb{P}(e(I_{1}, \dots , I_{k})\ge \tau \bar{\mu }(I_{1}, \dots , I_{k}))
\leq \exp \left ( -\frac{1}{2} \bar{\mu }(I_{1}, \dots , I_{k}) \tau
\log \tau \right ).
\end{align}
For a given number $c_{3}>0$, define $\gamma (I_{1},\dots ,I_{k})$ to be
the unique value of $\gamma $ such that
%
\begin{align}
\label{eq:defgamma}
\gamma \log \gamma =
\frac{c_{3} |I_{k}|}{\bar{\mu }(I_{1}, \dots , I_{k})} \log \left (
\frac{n}{|I_{k}|} \right ).
\end{align}
Let
$\kappa (I_{1}, \dots , I_{k}) = \max \{ 8, \gamma (I_{1}, \dots , I_{k})
\}$. Then by~\eqref{eq:chernoffbound} and~\eqref{eq:defgamma},
%
\begin{align}
&\mathbb{P}(e(I_{1}, \dots , I_{k})\ge \gamma (I_{1}, \dots , I_{k})
\bar{\mu }(I_{1}, \dots , I_{k}))
\notag
\\
&\le \exp \left ( -\frac{1}{2} \bar{\mu }(I_{1}, \dots , I_{k}) \kappa (I_{1},
\dots , I_{k}) \log \kappa (I_{1}, \dots , I_{k}) \right )
\notag
\\
&\le \exp \left [ -\frac{1}{2} c_{3} |I_{k}| \log \left (
\frac{n}{|I_{k}|} \right ) \right ].
\label{eq:cherbound2}
\end{align}
Let
$\Omega = \{(I_{1}, \dots , I_{k}): |I_{1}|\le \dots \le |I_{k}|\le
\frac{n}{e}\}$ and
\begin{equation*}
S(h_{1},\dots ,h_{k}):=\{(I_{1},\dots ,I_{k}):\forall i\in [k], |I_{i}|=h_{i}
\}.
\end{equation*}
By a union bound and~\eqref{eq:cherbound2}, we have
\begin{align*}
&\mathbb{P}\Big ( \exists (I_{1}, \dots , I_{k})\in \Omega : e(I_{1},
\dots , I_{k})\ge \gamma (I_{1}, \dots , I_{k})\bar{\mu }(I_{1},
\dots , I_{k}) \Big )
\\
& \le \sum _{(I_{1}, \dots , I_{k})\in \Omega } \exp \left [ -\frac{1}{2} c_{3}
|I_{k}| \log \left ( \frac{n}{|I_{k}|} \right ) \right ]
\\
& = \sum _{1\le h_{1}\le \dots \le h_{k}\le \frac{n}{e}} ~\sum _{(I_{1},
\dots , I_{k}) \in S(h_{1},\dots , h_{k})}\exp \Big [ -\frac{1}{2} c_{3} h_{k}
\log \Big ( \frac{n}{h_{k}} \Big ) \Big ]
\\
& = \sum _{1\le h_{1}\le \dots \le h_{k}\le \frac{n}{e}} {
\binom{n}{h_{1}}} \dots {\binom{n}{h_{k}}} \exp \Big [ -\frac{1}{2} c_{3}
h_{k} \log \Big ( \frac{n}{h_{k}} \Big ) \Big ].
\end{align*}

Since ${\binom{n}{k}} \leq (\frac{ne}{k})^{k}$ for any integer
$1\leq k\leq n$, we have the last line above is bounded by
\begin{align*}
& \sum _{1\le h_{1}\le \dots \le h_{k}\le \frac{n}{e}} \Big (
\frac{ne}{h_{1}} \Big )^{h_{1}} \dots \Big ( \frac{ne}{h_{k}} \Big )^{h_{k}}
\exp \Big [ -\frac{1}{2} c_{3} h_{k} \log \Big ( \frac{n}{h_{k}} \Big )
\Big ]
\\
\le & \sum _{1\le h_{1}\le \dots \le h_{k}\le \frac{n}{e}} \exp \Big [ -
\frac{1}{2} c_{3} h_{k} \log \Big ( \frac{n}{h_{k}} \Big ) + kh_{k}\log
\Big ( \frac{n}{h_{k}}\Big ) + kh_{k}\Big ]
\\
\le & \sum _{1\le h_{1}\le \dots \le h_{k}\le \frac{n}{e}} \exp \Big [ -
\frac{1}{2} (c_{3} - 4k) h_{k} \log \Big ( \frac{n}{h_{k}} \Big ) \Big ]
\\
\leq & \sum _{1\le h_{1}\le \dots \le h_{k}\le \frac{n}{e}} \exp \left [-
\frac{1}{2}(c_{3}-4k)\log n\right ]
\\
= &\sum _{1\le h_{1}\le \dots \le h_{k}\le \frac{n}{e}} n^{-\frac{1}{2} (c_{3}
-4k)}\le n^{-\frac{1}{2} (c_{3} -6k)}.
\end{align*}
As a result,
$e(I_{1}, \dots , I_{k})< \kappa (I_{1}, \dots , I_{k}) \bar{\mu }(I_{1},
\dots , I_{k})$ for all $(I_{1}, \dots , I_{k})\in \Omega $ with probability
at least $1- n^{-\frac{1}{2} (c_{3} -6k)}$. For any $r>0$, we can choose
$c_{3}$ large enough depending linearly on $k,r$ such that
$ 1- n^{-\frac{1}{2} (c_{3} -6k)}\leq 1-n^{-r}$. Suppose
$\kappa (I_{1}, \dots , I_{k}) = 8$, then
$e(I_{1}, \dots , I_{k})< 8 \bar{\mu }(I_{1}, \dots , I_{k})$ as desired.
Otherwise suppose
$ \kappa (I_{1}, \dots , I_{k}) =\gamma (I_{1}, \dots , I_{k})>8$, then
\begin{equation*}
\frac{e(I_{1}, \dots , I_{k})}{\bar{\mu }(I_{1}, \dots , I_{k})} <
\gamma (I_{1}, \dots , I_{k}).
\end{equation*}
Since $x\mapsto x\log x$ is an increasing function for $x\ge 1$, we have
\begin{align*}
\frac{e(I_{1}, \dots , I_{k})}{\bar{\mu }(I_{1}, \dots , I_{k})} \log
\frac{e(I_{1}, \dots , I_{k})}{\bar{\mu }(I_{1}, \dots , I_{k})} < &
\gamma (I_{1}, \dots , I_{k}) \log \gamma (I_{1}, \dots , I_{k})
\\
= &\frac{c_{3} |I_{k}|}{\bar{\mu }(I_{1}, \dots , I_{k})} \log \left (
\frac{n}{|I_{k}|} \right ),
\end{align*}
which gives the desired result for Case~\eqref{case:2}.
\end{proof}

With Lemma~\ref{lem:bounddeg} and Lemma~\ref{lem:bounddiscrepancy}, we
prove our estimates~\eqref{eq:heavy} for all heavy tuples. Recall we are
dealing with the tuples over $\overline{\mathcal{L}}^{+}$, we then have
%
\begin{align}
&\sum _{(i_{1},\dots ,i_{k})\in \overline{\mathcal{L}}^{+}} y_{1,i_{1}}
\cdots y_{k,i_{k}}\cdot t_{i_{1},\dots ,i_{k}}
\notag
\\
\leq & \sum _{
\substack{(s_{1},\dots ,s_{k}):\\2^{s_{1}+\cdots + s_{k}}\geq \sqrt{np}(2/\delta )^{k}n^{k/2-1}}}e(D_{1}^{s_{1}},
\dots , D_{k}^{s_{k}})\frac{2^{s_{1}}\delta }{\sqrt{n}}\cdots
\frac{2^{s_{k}}\delta }{\sqrt{n}}
\notag
\\
\leq & \delta ^{k}\sqrt{np} \sum _{
\substack{(s_{1},\dots ,s_{k}):\\2^{s_{1}+\cdots + s_{k}}\geq \sqrt{np}n^{k/2-1}}}
\alpha _{1,s_{1}}\cdots \alpha _{k,s_{k}}\sigma _{s_{1},\dots ,s_{k}}.
\label{eq:changeindex}
\end{align}
The last equality follows directly from definitions in~\eqref{def:alphasigma}.~\eqref{eq:changeindex} implies that it suffices
to estimate the contribution of heavy tuples through its index sets. We
then bound the contribution of heavy tuples by splitting the indices
$(s_{1},\dots ,s_{k})$ into 6 different categories. Let
%
\begin{align}
\label{eq:C}
\mathcal{C}:=\left \{  (s_{1},\cdots , s_{k}): 2^{s_{1}+\cdots + s_{k}}
\geq \sqrt{np} n^{k/2-1}, |D_{1}^{s_{1}}|\leq \cdots \leq |D_{k}^{s_{k}}|
\right \}
\end{align}
be the ordered index set for heavy tuples where we assume
$|D_{1}^{s_{1}}|\leq \cdots \leq |D_{k}^{s_{k}}|$. For the case where the
sequence $\{|D_{i}^{s_{i}}|,1\leq i\leq k\}$ have different orders can
be treated similarly, and there are $k!$ many in total. We then define
the following 6 categories in $\mathcal{C}$:
\begin{align*}
\mathcal{C}_{1} &=\left \{  (s_{1},\dots ,s_{k})\in \mathcal{C}:\sigma _{s_{1},
\dots ,s_{k}}\leq 1 \right \}  ,
\\
\mathcal{C}_{2} &=\left \{  (s_{1},\dots ,s_{k})\in \mathcal{C}\setminus
\mathcal{C}_{1}:\lambda _{s_{1},\dots ,s_{k}}\leq ec_{2} \right \}  ,
\\
\mathcal{C}_{3} &=\left \{  (s_{1},\dots ,s_{k})\in \mathcal{C}
\setminus (\mathcal{C}_{1}\cup \mathcal{C}_{2}): 2^{s_{1}+s_{2}+\cdots +s_{k-1}-s_{k}}
\geq n^{k/2-1}\sqrt{np} \right \}  ,
\\
\mathcal{C}_{4} &= \{ (s_{1},\dots ,s_{k})\in \mathcal{C}\setminus (
\mathcal{C}_{1}\cup \mathcal{C}_{2} \cup \mathcal{C}_{3}): \log \lambda _{s_{1},
\dots ,s_{k}}>\frac{1}{2} s_{k}\log 2+\frac{1}{4}\log (\alpha _{k,s_{k}}^{-1})
\},
\\
\mathcal{C}_{5} &=\left \{  (s_{1},\dots ,s_{k})\in \mathcal{C}
\setminus (\mathcal{C}_{1}\cup \mathcal{C}_{2} \cup \mathcal{C}_{3}\cup
\mathcal{C}_{4}):2s_{k}\log 2\geq \log (1/\alpha _{k,s_{k}}) \right \}  ,
\\
\mathcal{C}_{6} &= \mathcal{C}\setminus (\mathcal{C}_{1}\cup \mathcal{C}_{2}
\cup \mathcal{C}_{3}\cup \mathcal{C}_{4}\cup \mathcal{C}_{5}).
\end{align*}
For the rest of the proof, we will show for all $6$ categories
$\{\mathcal{C}_{t}, 1\leq t\leq 6\}$,
%
\begin{align}
\label{eq:heavytuplebound}
\sum _{ (s_{1},\dots ,s_{k})\in \mathcal{C}} \alpha _{1,s_{1}}\cdots
\alpha _{k,s_{k}}\sigma _{s_{1},\dots ,s_{k}}\mathbf{1}\{ (s_{1},
\dots , s_{k})\in \mathcal{C}_{t}\}\leq C_{k} \log ^{k-2}(n)
\end{align}
for some constant $C_{k}$ depending only on $k,c_{1},c_{2},c_{3}$ and
$\delta $, where the constants $c_{1},c_{2},c_{3}$ are the same ones as
in Lemma~\ref{lem:bounddeg} and Lemma~\ref{lem:bounddiscrepancy}. Here
$C_{k}$ depends exponentially on $k$ and linearly on $r$. We will repeatedly
use the following estimate which follows from~\eqref{eq:defalpha}:
%
\begin{align}
\label{eq:repeatinequality}
\sum _{s_{i}=1}^{\left \lceil \log _{2} (\sqrt{n}/\delta ) \right
\rceil }\alpha _{i,s_{i}}\leq \sum _{j\in [n]}|2y_{i,j}/\delta |^{2}
\leq (2/\delta )^{2}, \quad \forall 1\leq i\leq k.
\end{align}

From now on, for simplicity, whenever we are summing over $s_{i}$ for some
$1\leq i\leq k$, the range of $s_{i}$ is understood as
$1\leq s_{i}\leq \left \lceil \log _{2} (\sqrt{n}/\delta ) \right
\rceil $.

\subsubsection*{Tuples in $\mathcal{C}_{1}$}

In this case we get
\begin{align*}
&\sum _{(s_{1},\dots ,s_{k})\in \mathcal{C}}\alpha _{1,s_{1}}\cdots
\alpha _{k,s_{k}} \sigma _{s_{1},\dots , s_{k}}\mathbf{1}\{ (s_{1},
\dots , s_{k})\in \mathcal{C}_{1}\}
\\
\leq &\sum _{(s_{1},\dots , s_{k})\in \mathcal{C}}\alpha _{1,s_{1}}
\cdots \alpha _{k,s_{k}} \leq (2/\delta )^{2k} ,
\end{align*}
where the last inequality comes from~\eqref{eq:repeatinequality}.

\subsubsection*{Tuples in $\mathcal{C}_{2}$}

The constraint on $\mathcal{C}_{2}$ is the same as the condition in Case~\eqref{case:1} of Lemma~\ref{lem:bounddiscrepancy}. Recall Definition~\eqref{def:alphasigma} and~\eqref{eq:C}. We have
\begin{align*}
\sigma _{s_{1},\dots , s_{k}}& =\lambda _{s_{1},\dots , s_{k}}n^{k/2-1}
\sqrt{np}\cdot 2^{-(s_{1}+\cdots +s_{k})}\leq \lambda _{s_{1},\dots , s_{k}}
\leq e c_{2}.
\end{align*}
Therefore,
\begin{align*}
& \sum _{(s_{1},\dots , s_{k})\in \mathcal{C}}\alpha _{1,s_{1}}\cdots
\alpha _{k,s_{k}}\sigma _{s_{1},\dots , s_{k}} \mathbf{1}\left \{  (s_{1},
\dots , s_{k})\in \mathcal{C}_{2}\right \}
\\
\leq & e{c_{2}} \sum _{s_{1},\dots , s_{k}}\alpha _{1,s_{1}}\cdots
\alpha _{k,s_{k}}\leq ec_{2}(2/\delta )^{2k} .
\end{align*}

\subsubsection*{Tuples in $\mathcal{C}_{3}$}

By Lemma~\ref{lem:bounddeg}, all $(k-1)$-tuples have bounded degrees. Therefore
we have
\begin{align*}
e(D_{1}^{s_{1}},\dots ,D_{k}^{s_{k}})\leq c_{1}|D_{1}^{s_{1}}|\cdots |D_{k-1}^{s_{k-1}}|
np.
\end{align*}
Hence by Definition~\eqref{def:lambdamu},
%
\begin{align}
\label{eq:lambdaupperbound}
\lambda _{s_{1},\dots ,s_{k}}=
\frac{e(D_{1}^{s_{1}},\dots , D_{k}^{s_{k}})}{p |D_{1}^{s_{1}}|\cdots |D_{k}^{s_{k}}| }
\leq \frac{c_{1}n}{ |D_{k}^{s_{k}}|}.
\end{align}
Therefore we have
%
\begin{align}
& \sum _{(s_{1},\dots , s_{k})\in \mathcal{C}}\alpha _{1,s_{1}}\cdots
\alpha _{k,s_{k}}\sigma _{s_{1},\dots , s_{k}} \mathbf{1}\left \{  (s_{1},
\dots , s_{k})\in \mathcal{C}_{3}\right \}
\notag
\\
= &\sum _{(s_{1},\dots , s_{k})\in \mathcal{C}_{3}}\alpha _{1,s_{1}}
\alpha _{2,s_{2}}\cdots \alpha _{k,s_{k}}\lambda _{s_{1},\dots ,s_{k}}n^{k/2-1}
\sqrt{np}\cdot 2^{-(s_{1}+\cdots +s_{k})}
\notag
\\
\leq & \sum _{(s_{1},\dots ,s_{k})\in \mathcal{C}_{3}}\alpha _{1,s_{1}}
\cdots \alpha _{k-1,s_{k-1}}\frac{|D_{k}^{s_{k}}|2^{2s_{k}}}{n}
\frac{c_{1} n}{ |D_{k}^{s_{k}}|}n^{k/2-1}\sqrt{np} \cdot 2^{-(s_{1}+
\cdots + s_{k})}
\notag
\\
= &c_{1} n^{k/2-1}\sqrt{np}\sum _{(s_{1},\dots ,s_{k})\in \mathcal{C}_{3}}
\alpha _{1,s_{1}}\cdots \alpha _{k-1,s_{k-1}} 2^{s_{k}-(s_{2}+\cdots +s_{k-1})}
\notag
\\
= &c_{1}\sum _{s_{1},\dots ,s_{k}}\alpha _{1,s_{1}}\cdots \alpha _{k-1,s_{k-1}}n^{k/2-1}
\sqrt{np}\cdot 2^{s_{k}-(s_{2}+\cdots +s_{k-1})}\mathbf{1}\left \{  (s_{1},
\dots , s_{k})\in \mathcal{C}_{3}\right \}  ,
\label{eq:C3}
\end{align}
where the inequality in the third line is from~\eqref{eq:lambdaupperbound}. Since for all
$(s_{1},\dots ,s_{k})\in \mathcal{C}_{3}$ we have
$n^{k/2-1}\sqrt{np}\cdot 2^{s_{k}-(s_{2}+\cdots +s_{k-1})}\leq 1$, it implies
that
\begin{align*}
\sum _{s_{k}}n^{k/2-1}\sqrt{np}\cdot 2^{s_{k}-(s_{1}+\cdots +s_{k-1})}
\mathbf{1}\left \{  (s_{1},\dots ,s_{k})\in \mathcal{C}_{3}\right \}
\leq \sum _{i=0}^{\infty }2^{-i}\leq 2.
\end{align*}
In view of~\eqref{eq:repeatinequality}, we can bound~\eqref{eq:C3} by
\begin{equation*}
2c_{1}\sum _{s_{1},\dots ,s_{k-1}}\alpha _{1,s_{1}}\cdots \alpha _{k-1,s_{k-1}}
\leq 2c_{1}(2/\delta )^{2k-2}.
\end{equation*}
This completes the proof for the case of $\mathcal{C}_{3}$. For the remaining
categories $\mathcal{C}_{4},\mathcal{C}_{5}$ and $\mathcal{C}_{6}$, we rely
on the Case~\eqref{case:2} in the bounded discrepancy lemma. Recall
$\mathcal{C}_{2}$ corresponds to Case~\eqref{case:1} in Lemma~\ref{lem:bounddiscrepancy}. Therefore Case~\eqref{case:2} must hold in
$\mathcal{C}_{4},\mathcal{C}_{5}$ and $\mathcal{C}_{6}$. Case~\eqref{case:2} in Lemma~\ref{lem:bounddiscrepancy} can be written as
\begin{align*}
\lambda _{s_{1},\dots ,s_{k}}|D_{1}^{s_{1}}|\cdots |D_{k}^{s_{k}}|
\cdot p\log \lambda _{s_{1},\dots ,s_{k}}\leq c_{3}|D_{k}^{s_{k}}|
\log \left (\frac{n}{|D_{k}^{s_{k}}|}\right ).
\end{align*}
By definitions in~\eqref{def:alphasigma}, the inequality above is equivalent
to
%
\begin{align}
\label{eq:discrepancy2}
&\alpha _{1,s_{1}}\cdots \alpha _{k-1,s_{k-1}} \sigma _{s_{1},\dots ,s_{k}}
\log \lambda _{s_{1},\dots , s_{k}}
\notag
\\
\leq & c_{3}
\frac{2^{s_{1}+\cdots +s_{k-1}-s_{k}}}{n^{k/2-1}\sqrt{np}} \left (2s_{k}
\log 2+\log \alpha _{k,s_{k}}^{-1}\right ).
\end{align}
For the remaining of our proof, we will repeatedly use~\eqref{eq:discrepancy2}.

\subsubsection*{Tuples in $\mathcal{C}_{4}$}

The inequality
$ \log \lambda _{s_{1},\dots ,s_{k}}>\frac{1}{4}\left ( 2s_{k}\log 2+
\log (1/\alpha _{k,s_{k}})\right ) $ in the assumption of
$\mathcal{C}_{4}$ and~\eqref{eq:discrepancy2} imply that
\begin{align*}
\alpha _{1,s_{1}}\cdots \alpha _{k-1,s_{k-1}} \sigma _{s_{1},\dots ,s_{k}}
\leq 4c_{3}n^{1-k/2}\cdot 2^{s_{1}+\cdots + s_{k-1}-s_{k}}/\sqrt{np}.
\end{align*}
Then we have
%
\begin{align}
&\sum _{(s_{1},\dots , s_{k})\in \mathcal{C}}\alpha _{1,s_{1}}\cdots
\alpha _{k,s_{k}}\sigma _{s_{1},\dots , s_{k}} \mathbf{1}\{(s_{1},
\dots ,s_{k})\in \mathcal{C}_{4}\}
\notag
\\
=&\sum _{s_{k}}\alpha _{k,s_{k}}\sum _{s_{1},\dots ,s_{k-1}}\alpha _{1,s_{1}}
\cdots \alpha _{k-1,s_{k-1}}\sigma _{s_{1},\dots ,s_{k}}\mathbf{1}\{(s_{1},
\dots ,s_{k})\in \mathcal{C}_{4}\}
\notag
\\
\leq &4c_{3} \sum _{s_{k}}\alpha _{k,s_{k}}\sum _{s_{1},\dots ,s_{k-1}}
\frac{2^{s_{1}+\cdots + s_{k-1}-s_{k}}}{n^{k/2-1}\sqrt{np}}\mathbf{1}
\{(s_{1},\dots ,s_{k})\in \mathcal{C}_{4}\}.
\label{eq:geometric}
\end{align}
Since $(s_{1},\dots ,s_{k})\not \in \mathcal{C}_{3}$, we have
$ \frac{2^{s_{1}+\cdots + s_{k-1}-s_{k}}}{n^{k/2-1}\sqrt{np}}\leq 1$ for
all $(s_{1},\dots ,s_{k})\in \mathcal{C}_{4}$. Therefore~\eqref{eq:geometric} is bounded by
%
\begin{align}
&4c_{3} \sum _{s_{k}}\alpha _{k,s_{k}}\sum _{s_{1},\dots ,s_{k-2}}
\sum _{s_{k-1}}
\frac{2^{s_{1}+\cdots + s_{k-1}-s_{k}}}{n^{k/2-1}\sqrt{np}}
\mathbf{1}\{(s_{1},\dots ,s_{k})\in \mathcal{C}_{4}\}
\notag
\\
\leq & 4c_{3} \sum _{s_{k}}\alpha _{k,s_{k}}\sum _{s_{1},\dots ,s_{k-2}}
2\leq 8c_{3}\sum _{s_{k}}\alpha _{k,s_{k}} (\log _{2}(\sqrt{n}/\delta )+1)^{k-2},
\label{eq:extralog}
\end{align}
where the last inequality is from the fact that $s_{i}$ satisfies
$1\leq s_{i}\leq \left \lceil \log _{2} (\sqrt{n}/\delta ) \right
\rceil $ for $i\in [k]$ (see~\eqref{eq:rangeS}). Therefore~\eqref{eq:extralog} can be bounded by
%
\begin{align}
\label{eq:Clog}
8c_{3}\left (\frac{1}{2}\log _{2} n -\log _{2}(\delta )+1\right )^{k-2}
(2/\delta )^{2}\leq C\log ^{k-2}(n)
\end{align}
for a constant $C$ depending only on $\delta ,k$ and $c_{3}$.

\subsubsection*{Tuples in $\mathcal{C}_{5}$}

In this case we have
$ 2s_{k}\log 2\geq \log (\alpha _{k,s_{k}}^{-1}) $. Also because
$(s_{1},\dots , s_{k})\notin \mathcal{C}_{4}$, we have
%
\begin{align}
\label{eq:C5bound}
\log \lambda _{s_{1},\dots ,s_{k}}\leq \frac{1}{4}\left ( 2s_{k}\log 2+
\log (1/\alpha _{k,s_{k}})\right )\leq s_{k}\log 2,
\end{align}
thus $\lambda _{s_{1},\dots ,s_{k}}\leq 2^{s_{k}}$. On the other hand,
because $(s_{1},\dots , s_{k})\not \in \mathcal{C}_{1}$,
\begin{align*}
1<\sigma _{s_{1},\dots ,s_{k}}=\lambda _{s_{1},\dots ,s_{k}}n^{k/2-1}
\sqrt{np}\cdot 2^{-(s_{1}+\cdots +s_{k})}\leq n^{k/2-1}\sqrt{np}
\cdot 2^{-(s_{1}+\cdots + s_{k-1})}.
\end{align*}
Therefore we have
%
\begin{align}
\label{eq:s_1_to_s_k}
2^{s_{1}+\cdots + s_{k-1}}\leq n^{k/2-1} \sqrt{np} .
\end{align}
In addition, since $(s_{1},\dots , s_{k})\not \in \mathcal{C}_{2}$, we have
$\lambda _{s_{1},\dots ,s_{k}}>e c_{2}>e$, which implies
$\log \lambda _{s_{1},\dots ,s_{k}}\geq 1 $. Recall~\eqref{eq:discrepancy2}, together with~\eqref{eq:C5bound}, we then have
\begin{align*}
\alpha _{1,s_{1}}\cdots \alpha _{k-1,s_{k-1}} \sigma _{s_{1},\dots ,s_{k}}
&\leq \alpha _{1,s_{1}}\cdots \alpha _{k-1,s_{k-1}}\sigma _{s_{1},
\dots ,s_{k}}\log \lambda _{s_{1},\dots , s_{k}}
\\
&\leq c_{3}
\frac{2^{s_{1}+\cdots +s_{k-1}-s_{k}}}{n^{k/2-1}\sqrt{np}} \left (2s_{k}
\log 2+\log \alpha _{k,s_{k}}^{-1}\right )
\\
&\leq 4c_{3}\log 2\cdot s_{k}
\frac{2^{s_{1}+\cdots +s_{k-1}-s_{k}}}{ n^{k/2-1}\sqrt{np}} .
\end{align*}
Therefore,
%
\begin{align}
&\sum _{(s_{1},\dots ,s_{k})\in \mathcal{C}}\alpha _{1,s_{1}}\cdots
\alpha _{k,s_{k}}\sigma _{s_{1},\dots ,s_{k}}\mathbf{1}\left \{  (s_{1},
\dots s_{k})\in \mathcal{C}_{5} \right \}
\notag
\\
=&\sum _{s_{k}}\alpha _{k,s_{k}}\sum _{s_{1},\dots ,s_{k-1}}\alpha _{1,s_{1}}
\cdots \alpha _{k-1,s_{k-1}}\sigma _{s_{1},\dots ,s_{k}}\mathbf{1}
\left \{  (s_{1},\dots s_{k})\in \mathcal{C}_{5} \right \}
\notag
\\
\leq &\sum _{s_{k}}\alpha _{k,s_{k}}\sum _{s_{1},\dots ,s_{k-1}}4c_{3}
\log 2\cdot s_{k}
\frac{2^{s_{1}+\cdots +s_{k-1}-s_{k}}}{n^{k/2-1}\sqrt{np}}\mathbf{1}
\left \{  (s_{1},\dots s_{k})\in \mathcal{C}_{5} \right \}
\notag
\\
\leq &4c_{3}\sum _{s_{k}}\alpha _{k,s_{k}}\cdot s_{k}2^{-s_{k}}\sum _{s_{1},
\dots ,s_{k-1}}\frac{2^{s_{1}+\cdots +s_{k-1}}}{n^{k/2-1}\sqrt{np}}
\mathbf{1}\left \{  (s_{1},\dots s_{k})\in \mathcal{C}_{5} \right \}  .
\label{eq:C5}
\end{align}
From~\eqref{eq:s_1_to_s_k}, we have
$ \frac{2^{s_{1}+\cdots +s_{k-1}}}{n^{k/2-1}\sqrt{np}}\leq 1$ for any
$(s_{1},\dots ,s_{k})\in \mathcal{C}_{5}$. Note that
$s_{k}\cdot 2^{-s_{k}}\leq \frac{1}{2}$, therefore there exists a constant
$C$ depending only on $\delta ,k$ and $c_{3}$ such that~\eqref{eq:C5} can
be bounded by
\begin{align*}
& 2c_{3} \cdot \sum _{s_{k}}\alpha _{k,s_{k}} \sum _{s_{1},\dots ,s_{k-1}}
\frac{2^{s_{1}+\cdots +s_{k-1}}}{n^{k/2-1}\sqrt{np}}\mathbf{1}\left
\{  (s_{1},\dots s_{k})\in \mathcal{C}_{5} \right \}
\\
\leq & 2c_{3} (2/\delta )^{2} (\log _{2}(\sqrt{n}/\delta )+1)^{k-2}
\leq C\log ^{k-2}(n),
\end{align*}
where the inequality above follows in the same way as in~\eqref{eq:extralog} and~\eqref{eq:Clog}.

\subsubsection*{Tuples in $\mathcal{C}_{6}$}

In this case we have $2s_{k}\log 2<\log (\alpha _{k,s_{k}}^{-1}) $. Because
$(s_{1},\dots s_{k})\not \in \mathcal{(}C_{4}\cup \mathcal{C}_{2})$, we have
\begin{align*}
1\leq \log \lambda _{s_{1},\dots ,s_{k}}\leq \frac{1}{4}\left [ 2s_{k}
\log 2+\log (1/\alpha _{k,s_{k}})\right ]\leq \frac{1}{2}\log \alpha _{k,s_{k}}^{-1}
\leq \log \alpha _{k,s_{k}}^{-1},
\end{align*}
which implies
$ \lambda _{s_{1},\dots ,s_{k}}\alpha _{k,s_{k}}\leq 1$. Recall Definition~\eqref{def:alphasigma}. We obtain
%
\begin{align}
&\sum _{(s_{1},\dots ,s_{k})\in \mathcal{C}}\alpha _{1,s_{1}}\cdots
\alpha _{k,s_{k}}\sigma _{s_{1},\dots ,s_{k}}\mathbf{1}\left \{  (s_{1},
\dots , s_{k})\in \mathcal{C}_{6} \right \}
\notag
\\
=&\sum _{(s_{1},\dots ,s_{k-1},s_{k})\in \mathcal{C}_{6}}\alpha _{1,s_{1}}
\cdots \alpha _{k-1,s_{k-1}}\alpha _{k,s_{k}}\lambda _{s_{1},\dots ,s_{k}}n^{k/2-1}
\sqrt{np}\cdot 2^{-(s_{1}+\cdots +s_{k})}
\notag
\\
\leq &\sum _{s_{1},\dots ,s_{k-1}}\alpha _{1,s_{1}}\cdots \alpha _{k-1,s_{k-1}}
\sum _{s_{k}}n^{k/2-1}\sqrt{np}\cdot 2^{-(s_{1}+\cdots +s_{k})}
\mathbf{1}\left \{  (s_{1},\dots , s_{k})\in \mathcal{C}_{6} \right \}  .
\label{eq:C6}
\end{align}
Recall from~\eqref{eq:C},
$2^{s_{1}+\cdots + s_{k}}\geq \sqrt{np}\cdot n^{k/2-1}$, we have
\begin{equation*}
\sqrt{np}\cdot 2^{-(s_{1}+\cdots +s_{k})}\leq n^{1-k/2}.
\end{equation*}
for all $(s_{1},\dots ,s_{k})\in \mathcal{C}_{5}$. Hence
\begin{equation*}
\sum _{s_{k}}n^{k/2-1}\sqrt{np}\cdot 2^{-(s_{1}+\cdots +s_{k})}
\mathbf{1}\left \{  (s_{1},\dots , s_{k})\in \mathcal{C}_{6} \right \}
\leq 2.
\end{equation*}
Therefore~\eqref{eq:C6} can be bounded by
\begin{equation*}
2 \sum _{s_{1},\dots ,s_{k-1}}\alpha _{1,s_{1}}\cdots \alpha _{k-1,s_{k-1}}
\leq 2(2/\delta )^{2k-2}.
\end{equation*}
Combining all the estimates from $\mathcal{C}_{1}$ to
$\mathcal{C}_{6}$, we have~\eqref{eq:heavytuplebound} holds. This completes
the proof of Theorem~\ref{thm:lei}.

\section{Proof of Theorem \ref{thm:lowerboundSpec}}
\label{sec5}

Let $e_{1}=(1,0,\dots ,0)$ be a unit vector in $\mathbb{R}^{n}$ and denote
$W=T-\mathbb{E}T$. Define a matrix $A\in \mathbb{R}^{n\times n}$ such that
$a_{i_{1},i_{2}}=w_{i_{1},i_{2},1,\dots ,1}$. By the definition of the
spectral norm,
\begin{align*}
\|T-\mathbb{E}T\|&\geq \max _{\|x\|_{2}=\|y\|_{2}=1}| \langle W,x
\otimes y\otimes e_{1}\otimes \cdots \otimes e_{1}\rangle |
\\
&=\max _{\|x\|_{2}=\|y\|_{2}=1}\left | \sum _{i_{1},i_{2}}w_{i_{1},i_{2},1,
\dots ,1}\cdot x_{i_{1}}y_{i_{2}}\right |=\|A\|.
\end{align*}
Note that $A$ is an $n\times n$ sparse random matrix with independent centered
Bernoulli entries. Since the limiting singular value distribution of
$\frac{A}{\sqrt{np}}$ is known (see for example Theorem 3.6 in
\cite{bai2010spectral}), the largest singular value of $A$ is at least
$(2-o(1))\sqrt{np}$ with high probability. Therefore
$\|T-\mathbb{E}T\|\geq \sqrt{np}$ with high probability. This completes
the proof.

\section{Proof of Theorem \ref{thm:main}}
\label{sec:unfolding}

When $p=o(\log n/n)$, a direct application of the Kahn-Szemer{\'{e}}di argument
would fail since the bounded degree and bounded discrepancy properties
in the proof of Theorem~\ref{thm:lei} do not hold with high probability.
We will use the following operation called tensor unfolding, in the proof
of Theorem~\ref{thm:main}. For more details, see
\cite{kolda2009tensor,wang2017operator}.

\begin{definition}[Partition]
For any $n$, the $l$-partition $\pi $ of $[k]$ is a collection
$\{B_{1}^{\pi },\dots , B_{l}^{\pi } \}$ of $l$ disjoint non-empty subsets
$B_{i}^{\pi }, 1\leq i\leq l$ such that $\cup _{i=1}^{l} B_{i}^{\pi }=[k]$.
\end{definition}

\begin{definition}[Tensor unfolding]%
\label{def:tensorunfolding}
Let $T\in \mathbb{R}^{n\times \cdots \times n}$ be an order-$k$ tensor and
$\pi $ be an $l$-partition of $[k]$. The partition of $\pi $ defines a
map
\begin{align*}
\phi _{\pi }: & [n]^{k} \to \left [n^{|B_{1}^{\pi }|}\right ]\times
\cdots \times \left [n^{|B_{l}^{\pi }|}\right ],\quad &\phi _{\pi }(i_{1},
\dots ,i_{k})=(m_{1},\dots ,m_{l}),
\end{align*}
where
\begin{equation*}
m_{j}=1+\prod _{r\in B_{j}^{\pi}}(i_{r}-1)J_r, \quad J_r=\prod_{\substack{l\in B_j^{\pi},\\ l<r}} n.
\end{equation*}
The map $\phi _{\pi }$ induces an unfolding action
$T\to \textnormal{Unfold}_{\pi }(T)$, where
$\textnormal{Unfold}_{\pi }(T)$ is a tensor of order $l$ such that
$\textnormal{Unfold}_{\pi }(T)_{m_{1},\dots ,m_{l}}=T_{i_{1},\dots ,i_{k}}
$.
\end{definition}
We will use the following inequality between the spectral norms of the
original tensor $T$ and the unfolded tensor
$\textnormal{Unfold}_{\pi }(T)$.
%
\begin{lemma}[Proposition 4.1 in \cite{wang2017operator}]%
\label{lem:unfolding}
For any order-$k$ tensor $T$ and any partition $\pi $ of $[k]$, we have
$ \|T\| \leq \|\textnormal{Unfold}_{\pi }(T)\| $.
\end{lemma}

With Lemma~\ref{lem:unfolding}, we are ready to prove Theorem~\ref{thm:main}.

\begin{proof}[Proof of Theorem~\ref{thm:main}]
(1) Assume $p\geq \frac{c\log n}{n^{m}}$ with an integer $m$ such that
$k/2\leq m\leq k-1$. Consider a $2$-partition of $[k]$ denoted by
$\pi _{1}=\{ \{1,2,\dots ,m\}, \{ m+1,\dots , k\}\}$. From Lemma~\ref{lem:unfolding}, we have
%
\begin{align}
\label{eq:conT}
\|T-\mathbb{E}T\| \leq \|\textnormal{Unfold}_{\pi _{1}}(T-\mathbb{E}T)
\|.
\end{align}
Here $\textnormal{Unfold}_{\pi _{1}}(T-\mathbb{E}T)$ is an
$n^{k-m}\times n^{m}$ random matrix whose entries are one-to-one correspondent
to entries in $T-\mathbb{E}T$. Let
$A\in \mathbb{R}^{n^{m}}\times \mathbb{R}^{n^{m}}$ be a matrix such that
\begin{equation*}
A_{i,j}=
\begin{cases}
(\textnormal{Unfold}_{\pi _{1}}(T))_{i,j} & \text{ if } i
\in [n^{k-m}], j\in [n^{m}],
\\
0 & \text{ otherwise.}
\end{cases}
\end{equation*}
Then $A$ is an adjacency matrix of a random directed graph on
$n^{m}$ many vertices with
\begin{equation*}
p\geq \frac{c\log n}{n^{m}}=\frac{c}{m}\cdot
\frac{\log (n^{m})}{n^{m}}.
\end{equation*}
Then we apply Theorem~\ref{thm:lei} with the matrix case. For any
$r>0$, there is a constant $C>0$ depending on $r$ and $\frac{c}{m}$ such
that $\|A-\mathbb{E}A\| \leq C\sqrt{n^{m}p}$ with probability at least
$1-n^{-rm}$. Then from~\eqref{eq:conT}, with probability at least
$1-n^{-r}$,
\begin{equation*}
\|T-\mathbb{E}T\| \leq \|\textnormal{Unfold}_{\pi _{1}}(T-\mathbb{E}T)
\|\leq \|A-\mathbb{E}A\| \leq C\sqrt{n^{m}p}.
\end{equation*}
This completes the proof of~\eqref{eq:unfoldeq1}.

(2) Assume $p\geq \frac{c\log n}{n^{m}}$ with an integer $m$ such that
$1\leq m<k/2$. Denote $k=ml+s$ for some $l\geq 1, 0\leq s\leq m-1$.

When $s\not =0$, let $\pi _{2}$ be a $(l+1)$-partition of $[k]$ into $l+1$ parts
given by
\begin{equation*}
\pi _{2}=\{ \{1,\dots ,m\}, \{m+1,\dots , 2m\},\dots , \{ml+1,\dots ,ml+s
\} \}.
\end{equation*}
Then
$\textnormal{Unfold}_{\pi _{2}}(T-\mathbb{E}T)\in \mathbb{R}^{n^{m}}
\times \cdots \times \mathbb{R}^{n^{m}}\times \mathbb{R}^{n^{s}} $ is a tensor
of order $(l+1)$. Since
$p\geq \frac{c}{m} \frac{\log (n^{m})}{n^{m}}$, we can apply Theorem~\ref{thm:lei} and Lemma~\ref{lem:unfolding}. Then with probability at least
$1-n^{-r}$, we have
\begin{align*}
\| T-\mathbb{E}T \| &\leq \| \textnormal{Unfold}_{\pi _{2}}(T-
\mathbb{E}T)\|
\\
&\leq C\sqrt{n^{m} p}\log ^{l-1} (n^{m})\leq C_{1}\sqrt{n^{m}p}\log ^{
\frac{k-1}{m}-1}(n),
\end{align*}
where $C_{1}$ is a constant depending on $k,r,c$.

When $s=0$, we can similarly define $\pi _{2}$ as a $l$-partition of
$[k]$ into $l$ blocks such that
$\textnormal{Unfold}_{\pi _{2}}(T-\mathbb{E}T)\in \mathbb{R}^{n^{m}}
\times \cdots \times \mathbb{R}^{n^{m}} $ is a tensor of order $l$. With
probability at least $1-n^{-r}$, for a constant $C_{2}$ depending on
$k,r,c$, we have
\begin{align*}
\| T-\mathbb{E}T \|\leq C_{2}\sqrt{n^{m}p}\log ^{k/m -2}(n).
\end{align*}
This completes the proof of Theorem~\ref{thm:main}.
\end{proof}

\section{Proof of Theorem~\ref{thm:minimax}}
\label{sec:lowerbound}

In this section, we will prove Theorem~\ref{thm:minimax}. We first compute
the packing number over the parameter space under the spectral norm, then
apply Fano's inequality. We first introduce several useful lemmas for showing
this result. We will use the version in~\cite{Tsybakov2008introduction}.

\begin{lemma}[Varshamov-Gilbert bound]%
\label{lem:VG:bound}
For $n\ge 8$, there exists a subset $S\subset \{0,1\}^{n}$ such that
$|S|\ge 2^{n/8}+1$ and for every distinct pair of
$\omega ,\omega '\in S$, the Hamming distance satisfies
\begin{equation*}
H(\omega ,\omega '):=\|\omega -\omega '\|_{1}> n/8.
\end{equation*}
\end{lemma}

\begin{lemma}[Fano's inequality]%
\label{lem:fano}
Assume $N\ge 3$ and suppose
$\{\theta _{1}, \dots , \theta _{N}\}\subset \Theta $ such that
\begin{itemize}
\item[(i)] for all $1\le i<j\le N$,
$d(\theta _{i}, \theta _{j})\ge 2\alpha $, where $d$ is a metric on
$\Theta $;
\item[(ii)] let $P_{i}$ be the distribution with respect to parameter
$\theta _{i}$, then for all $i,j\in [N]$, $P_{i}$ is absolutely continuous
with respect to $P_{j}$;
\item[(iii)] for all $i,j\in N$, the Kullback-Leibler divergence
$D_{\textnormal{KL}} ({P_{i}} \| {P_{j}} )\le \beta \log (N-1)$ for some
$0<\beta <1/8$.
\end{itemize}
Then
\begin{align*}
\inf _{\hat{\theta }}\sup _{\theta \in \Theta } \mathbb{P}(d(\hat{\theta },
\theta )\ge \alpha )\ge \frac{\sqrt{N-1}}{1+\sqrt{N-1}}\left (1-2
\beta -\sqrt{\frac{2\beta }{\log (N-1)}}\right ).
\end{align*}
\end{lemma}

Since we will apply Fano's inequality associated with Kullback-Leibler
divergence, it requires the following lemma about random tensor with independent
Bernoulli entries.
%
\begin{lemma}%
\label{lem:kl:bound}
For $0\le a<b\le 1$, we consider parameters
$\theta , \theta '\in [a,b]^{n^{k}}$ for $0\le a<b\le 1$, and let
$P$ and $P'$ be the corresponding distributions, then the Kullback-Leibler
divergence satisfies
\begin{align*}
D_{\textnormal{KL}} ({P} \| {P'} ) \le
\frac{\|{\theta -\theta '}\|_{F}^{2}}{a(1-b)}.
\end{align*}
\end{lemma}

\begin{proof}
We firstly consider entrywise KL-divergence. For $p,q\in [a,b]$,
\begin{align*}
D_{\textnormal{KL}} ({\text{Ber}(p)} \| {\text{Ber}(q)} ) &= p\log
\frac{p}{q} + (1-p)\log \frac{1-p}{1-q}
\\
&= p\log \Big ( 1+\frac{p-q}{q} \Big ) + (1-p)\log \Big (1-\frac{p-q}{1-q}
\Big )
\\
&\le p\Big ( \frac{p-q}{q} \Big ) + (1-p)\Big (-\frac{p-q}{1-q}\Big ) =
\frac{(p-q)^{2}}{q(1-q)}\le \frac{(p-q)^{2}}{a(1-b)}.
\end{align*}
By independence of each entry, we have
$D_{\textnormal{KL}} ({P} \| {P'} ) \le
\frac{\|{\theta -\theta '}\|_{F}^{2}}{a(1-b)}$.
\end{proof}

Now we are ready to prove Theorem~\ref{thm:minimax}.

\begin{proof}[Proof of Theorem~\ref{thm:minimax}]
By Lemma~\ref{lem:VG:bound}, there exists a subset
$\{\omega ^{(1)}, \dots , \omega ^{(N)}\}$ of $\{0,1\}^{n}$ such that
\begin{equation*}
\min _{1\le i<j\le N} H(\omega ^{(i)}, \omega ^{(j)})>\frac{n}{8} \quad
\text{ and } \quad N\ge 2^{n/8}+1\ge e^{n/12}+1.
\end{equation*}
We note that
$H(\omega ^{(i)}, \omega ^{(j)})=\|\omega ^{(i)}- \omega ^{(j)}\|_{2}^{2}$.
Let $W$ be a fixed order-$(k-1)$ tensor with entries either $0$ or
$1$ and dimension $n\times \dots \times n$. The entries of $W$ is designed
as follows. Let $m=\lfloor p^{-\frac{1}{k-1}}\rfloor \wedge n$, so
$1\le m^{k-1}\le 1/p$. We assign $1$'s to an
$m\times \dots \times m$ subtensor of $W$ and assign $0$'s to the rest
entries. Then the rank of $W$ is 1 and $\|W\|=\|W\|_{F}=m^{(k-1)/2}$. Now
we define for $i\in [N]$,
\begin{align*}
\theta ^{(i)} := \frac{p}{2} J + \frac{p}{30} \omega ^{(i)}\otimes W,
\end{align*}
where $J\in \mathbb{R}^{n^{k}}$ is an order-$k$ tensor with all ones, and
\begin{equation*}
(\omega ^{(i)}\otimes W)_{i_{1},\dots , i_{k}} = \omega ^{(i)}_{i_{1}}
w_{i_{2},\dots , i_{k}}.
\end{equation*}
Then for all $i,j\in [N]$,
$\theta ^{(i)}-\theta ^{(j)}=\frac{p}{30}(\omega ^{(i)}-\omega ^{(j)})
\otimes W$. By the choice of $\theta ^{(i)}$'s,
\begin{align*}
\min _{1\le i<j\le N}\|\theta ^{(i)} - \theta ^{(j)}\|^{2} = \min _{1
\le i<j\le n}
\frac{\|\omega ^{(i)}- \omega ^{(j)}\|_{2}^{2} \|W\|^{2}p^{2}}{900}
\ge \frac{nm^{k-1}p^{2}}{7200}.
\end{align*}
On the other hand, $\|\omega ^{(i)}- \omega ^{(j)}\|_{2}^{2}\le n$, so
\begin{align*}
\max _{1\le i<j\le N}\|\theta ^{(i)} - \theta ^{(j)}\|^{2} = \max _{1
\le i<j\le N}
\frac{\|\omega ^{(i)}- \omega ^{(j)}\|_{2}^{2} \|W\|^{2}p^{2}}{900}
\le \frac{nm^{k-1}p^{2}}{900}.
\end{align*}
Let $P_{i}$ be the distribution of a random tensor $T$ associated with
parameter $\theta ^{(i)}$ for $i\in [N]$. Since
$\theta ^{(i)}\in [\frac{p}{2}, \frac{8p}{15}]^{n^{k}}$, by Lemma~\ref{lem:kl:bound},
we have
\begin{align*}
\max _{1\le i<j\le N} D_{\textnormal{KL}}(P_{i}\| P_{j})&\le \max _{1
\le i<j\le N}
\frac{\|\theta ^{(i)} - \theta ^{(j)}\|^{2}_{F}}{\big (\frac{p}{2}\big )\big (1-\frac{8p}{15}\big )}
\\
&\le
\frac{nm^{k-1}p^{2}}{900\big (\frac{p}{2}\big )\big (1-\frac{8p}{15}\big )}
\le \frac{nm^{k-1}p}{210}\le \frac{n}{210},
\end{align*}
where the last inequality is due to the choice
$m=\lfloor p^{-\frac{1}{k-1}}\rfloor \wedge n\le p^{-\frac{1}{k-1}}$. To
apply Fano's inequality, we let
$\alpha = \frac{nm^{k-1}p^{2}}{14400}$ and verify that for
$i,j\in [N]$,
\begin{equation*}
D_{\textnormal{KL}}(\theta ^{(i)}, \theta ^{(j)})\le \frac{n}{210}\le
\beta \log e^{n/12}
\end{equation*}
for $\beta =\frac{1}{9}$. Then by Lemma~\ref{lem:fano}, we have
\begin{align*}
\inf _{\hat{\theta }}\sup _{\theta \in [0,p]^{n^{k}}}\mathbb{P}\Big (\|
\hat{\theta }- \theta \|^{2}\ge \frac{nm^{k-1}p^{2}}{14400}\Big ) &\ge
\frac{2^{n/16}}{1+2^{n/16}}\left (1-\frac{2}{9} -\sqrt{\frac{2/9}{n/12}}
\right )\ge \frac{1}{3}
\end{align*}
when $n\ge 16$. By the choice of $m$, we have
\begin{equation*}
nm^{k-1}p^{2} =n(\lfloor p^{-\frac{1}{k-1}}\rfloor \wedge n)^{k-1}p^{2}
\ge (2^{1-k}np)\wedge (n^{k}p^{2}),
\end{equation*}
which gives the desired result.
\end{proof}

\section{Proof of Lemma~\ref{lem:num:regularized}}
\label{sec8}

Similar to~\eqref{eq:boundedelemma}, by Bernstein's inequality, we have
for each $(i_{1},\dots ,i_{k-m})\in [n]^{k-m}$,
\begin{equation*}
\mathbb{P}(d_{i_{1},\dots ,i_{k-m}}>2 n^{m}p)\le \exp \left ( -
\frac{3n^{m}p}{8}\right ).
\end{equation*}
Then $\mathbf{1}\{ d_{i_{1},\dots ,i_{k-m}}>2n^{m}p\}$ is a Bernoulli random
variable with mean at most
$\mu :=\exp \left ( -\frac{3n^{m}p}{8}\right )$. Since
$d_{i_{1},\dots ,i_{k-m}}$ are independent for all
$(i_{1},\dots ,i_{k-m})\in [n]^{k-m}$, by Chernoff's inequality~\eqref{eq:weakercher}, for any $\lambda \geq 0$,
%
\begin{align}
\label{eq:bernstein}
& \mathbb{P}\left (|\tilde{S}|\geq (1+\lambda )n^{k-m}\mu \right )
\notag
\\
=& \mathbb{P}\left ( \sum _{i_{1},\dots ,i_{k-m}\in [n]}^{n}
\mathbf{1}\{ d_{i_{1},\dots ,i_{k-m}}>2n^{m}p\}\geq (1+\lambda )n^{k-m}
\mu \right )
\notag
\\
\leq &\exp \left ( -\frac{\lambda ^{2}n^{k-m}\mu }{2+\lambda }\right ).
\end{align}
Since $n^{m}p\ge c$, we can choose a constant $c=9$ and take
%
\begin{align}
\label{eq:deltabound}
\lambda =\frac{1}{n^{m}p\mu }-1=
\frac{\exp \left (\frac{3n^{m}p}{8}\right )}{n^{m}p}-1\ge 1,
\end{align}
so that $2+\lambda \le 3\lambda $, and from~\eqref{eq:deltabound} we know
%
\begin{align}
\label{eq:nexp}
n^{k-m}\exp \left (-\frac{3n^{m}p}{8}\right )\leq
\frac{1}{2n^{2m-k}p}.
\end{align}
Then~\eqref{eq:bernstein} implies
\begin{align}
\mathbb{P}\left (|S|\geq \frac{1}{n^{2m-k}p} \right ) &\leq \exp
\left ( -\frac{\lambda n^{k-m} \mu }{3} \right )
\notag
\\
&=\exp \left (-\frac{1}{3} n^{k-m}\mu \left (\frac{1}{n^{m}p\mu }-1
\right )\right )
\notag
\\
&= \exp \left ( -\frac{1}{3n^{2m-k}p}+\frac{1}{3}n^{k-m}\exp \left ( -
\frac{3n^{m}p}{8}\right )\right )
\notag
\\
&\leq \exp \left (-\frac{1}{6n^{2m-k}p} \right ) \leq \exp \left ( -
\frac{n^{k-m}}{6\log n}\right ),
\notag
\end{align}
where the last line of inequalities follows from~\eqref{eq:nexp} and our
assumption that $n^{m}p<\log n$.

\section{Proof of Theorem \ref{thm:regularize}}
\label{sec:regularization}

Let $T$ be the adjacency tensor of a $k$-uniform random directed hypergraph
and $P=\mathbb{E}T$. Recall Definition~\ref{def:tensorunfolding}. Let
$\pi =\{ \{1,2,\dots ,k-m\}, \{k-m+1,\dots , k\}\}$ be a $2$-partition of
$[k]$. Then $\textnormal{Unfold}_{\pi }(T-\mathbb{E}T)$ is an
$n^{k-m}\times n^{m}$ random matrix whose entries are one-to-one correspondent
to entries in $(T-\mathbb{E}T)$. Let
$A\in \mathbb{R}^{n^{m}}\times \mathbb{R}^{n^{m}}$ be a matrix such that
\begin{equation*}
A_{i,j}=
\begin{cases}
\textnormal{Unfold}_{\pi }(T)_{i,j} & \text{ if } i\in [n^{k-m}], j
\in [n^{m}],
\\
0 & \text{ otherwise.}
\end{cases}
\end{equation*}
Then $A$ is an adjacency matrix of a random directed graph on
$n^{m}$ vertices with $p\geq \frac{c}{n^{m}}$. Regularizing $A$ by removing
vertices of degrees greater than $2n^{m}p$, from Theorem 2.1 in
\cite{le2017concentration}, we have with probability at least
$1-n^{-rm}$, $ \|\hat{A}-\mathbb{E}A \|\leq C\sqrt{n^{m}p}$. By the way
we regularize an order-$k$ random tensor $T$ in~\eqref{eq:regularizationdef}, we have
$\textnormal{Unfold}_{\pi }(\hat{T}-P)$ is a submatrix of
$(\hat{A}-\mathbb{E}A)$ with other entries being $0$. Therefore by Lemma~\ref{lem:unfolding}, with probability at least $1-n^{-r}$,
\begin{equation*}
\|\hat{T}-P\|\leq \|\textnormal{Unfold}_{\pi }(\hat{T}-P)\|\leq \|
\hat{A}-\mathbb{E}A \|\leq C\sqrt{n^{m}p}.
\end{equation*}
This completes the proof of~\eqref{eq:regularization2} in Theorem~\ref{thm:regularize}.

\section{Proof of Theorem \ref{thm:mixing}}
\label{sec10}

Let $1_{V_{i}}$ be the indicator vector of
$V_{i}\not =\emptyset ,1\leq i\leq k$ such that the $j$-th entry of
$1_{V_{i}}$ is $1$ if $j\in V_{i}$ and $0$ if $j\not \in V_{i}$. We then
have
\begin{align*}
&
\frac{\left |e_{H'}(V_{1},\dots , V_{k})-p |V_{1}|\cdots |V_{k}| \right |}{\sqrt{|V_{1}|\cdots |V_{k}|}}
\\
=&\frac{\left |T'(1_{V_{1}},\dots , 1_{V_{k}})-p\cdot J(1_{V_{1}},\dots , 1_{V_{k}})\right |}{\sqrt{|V_{1}|\cdots |V_{k}|}}
\\
=&\left |T'\left (\frac{1_{V_{1}}}{\sqrt{|V_{1}|}},\dots ,
\frac{1_{V_{k}}}{\sqrt{|V_{k}|}}\right )-p\cdot J\left (
\frac{1_{V_{1}}}{\sqrt{|V_{1}|}},\dots ,
\frac{1_{V_{k}}}{\sqrt{|V_{k}|}}\right )\right |
\\
\leq &\left \|  T'-p J\right \|  \leq C\sqrt{n^{k-1}p}=C\sqrt{c},
\end{align*}
where the last line is from the definition of the spectral norm for tensors
and~\eqref{eq:expandreg}. Then~\eqref{eq:mixinglemma} follows.

\section{Conclusions}
\label{sec11}

In this paper, we considered the concentration of sparse random tensors
under spectral norm in different sparsity regimes. When
$p\geq \frac{c\log n}{n}$, the extra log factor in Theorem 2.1 is due
to the proof techniques, which only appears when bounding the contribution
from $\mathcal{C}_{4}$ and $\mathcal{C}_{5}$ in Section~\ref{sec:heavy}.
The Kahn-Szemer{\'{e}}di argument was specially designed to control random
quadratic forms in the matrix case, but not for random multi-linear forms
in the tensor case. An improvement to remove the extra log factors (which
we conjecture should not appear) will require a new argument.

The regularization step we used is a sufficient way to recover the concentration
of spectral norms, and it relies on the tensor unfolding inequality in
Lemma~\ref{lem:unfolding}. It remains to extend the analysis to other sparsity
regimes. Proving a lower bound on the spectral norm without regularization
will involve a more combinatorial argument similar to
\cite{krivelevich2003largest,benaych2019largest}. We leave it as a future
direction.

It would also be interesting to discuss the dependence on $k$ for the constant
$C$ in~\eqref{eq:Hadamardproductinequality}. Using the
$\varepsilon $-net argument, we obtain a constant $C$ depending exponentially
on $k$. It is possible that by using different arguments, the dependence
can be improved.

\subsection*{Acknowledgments}
We thank anonymous referees for their detailed comments and suggestions, which have improved the quality of this paper.  We also thank Arash A. Amini, Nicholas Cook,  Ioana Dumitriu, Kameron Decker Harris and Roman Vershynin for helpful comments. Y.Z. is partially supported by NSF DMS-1949617.

\bibliographystyle{plain}
\bibliography{ref.bib}

\end{document}